\documentclass[a4paper,12pt]{amsart}

\usepackage[ansinew]{inputenc}
\usepackage[english]{babel}
\usepackage[T1]{fontenc}
\usepackage{amsmath, amssymb, amsthm}
\usepackage[foot]{amsaddr}
\usepackage{csquotes, mathtools, bbm}
\usepackage{geometry}
\usepackage{enumerate}
\usepackage{hyperref}
\usepackage{todonotes}
\usepackage[style=alphabetic,backend=biber]{biblatex}

\numberwithin{equation}{section}
\allowdisplaybreaks

\newtheorem{satz}{Theorem}[section]
\newtheorem{proposition}[satz]{Proposition}
\newtheorem{lemma}[satz]{Lemma}
\newtheorem{korollar}[satz]{Corollary}

\theoremstyle{remark}

\newtheorem{bemerkung}[satz]{Remark}
\newtheorem*{beispiel*}{Example}
\newtheorem*{bemerkung*}{Remark}

\DeclarePairedDelimiter	{\abs}		{\lvert}	{\rvert}
\DeclarePairedDelimiter	{\babs}		{\Big\lvert}	{\Big\rvert}
\DeclarePairedDelimiter	{\bbrace}		{\Big\lbrace}	{\Big\rbrace}
\DeclarePairedDelimiter	{\norm}		{\lVert}	{\rVert}
\DeclarePairedDelimiter	{\skal}		{\langle}	{\rangle}

\DeclareMathOperator	{\IE}		{\mathbb{E}}
\DeclareMathOperator	{\IP}		{\mathbb{P}}
\DeclareMathOperator	{\IN}		{\mathbb{N}}
\renewcommand{\epsilon}{\varepsilon}

\title[Concentration of polynomials in $\alpha$-sub-exponential r.v.]{Concentration inequalities for polynomials in $\alpha$-sub-exponential random variables}

\author		{Friedrich G\"otze$^1$, Holger Sambale$^1$ and Arthur Sinulis$^1$}
\address	{$^1$Faculty of Mathematics, Bielefeld University, Bielefeld, Germany}

\email[A1]		{goetze@math.uni-bielefeld.de}
\email[A2]		{hsambale@math.uni-bielefeld.de}
\email[A3]		{asinulis@math.uni-bielefeld.de}

\begin{document}
\thanks{This research was supported by the German Research Foundation via CRC 1283.}
\begin{abstract}
	In this work we derive multi-level concentration inequalities for polynomial functions in independent random variables with a $\alpha$-sub-exponential tail decay.
A particularly interesting case is given by quadratic forms $f(X_1, \ldots, X_n) = \skal{X,A X}$, for which we prove Hanson--Wright-type inequalities with explicit dependence on various norms of the matrix $A$. A consequence of these inequalities is a two-level concentration inequality for quadratic forms in $\alpha$-sub-exponential random variables, such as quadratic Poisson chaos.
  
We provide various applications of these inequalities. Among these are generalizations the results given by Rudelson--Vershynin from sub-Gaussian to $\alpha$-sub-exponential random variables, i.\,e.\ concentration of the Euclidean norm of the linear image of a random vector, small ball probability estimates and concentration inequalities for the distance between a random vector and a fixed subspace. Moreover, we obtain concentration inequalities for the excess loss in a fixed design linear regression and the norm of a randomly projected random vector.
\end{abstract}
\date{\today}
\maketitle

\section{Introduction}
Let $X_1, \ldots, X_n$ be independent random variables and let $f: \mathbb{R}^n \to \mathbb{R}$ be a measurable function. One of the main and rather classical questions of probability theory consists in finding good estimates on the fluctuations of $f(X_1,\ldots,X_n)$ around a deterministic value (e.\,g. its expectation or median), i.\,e.\ to determine a function $h: [0,\infty) \to [0,1]$ such that
\begin{equation}
\mathbb{P}\Big( \abs{f(X_1,\ldots, X_n) - \IE f(X_1,\ldots,X_n)} \ge t \Big) \le h(t).
\end{equation}
Of course, $h$ should take into account both the information given by $f$ as well as $X_1, \ldots, X_n$. Perhaps one of the most well-known \emph{concentration inequalities} is the tail decay of the Gaussian distribution: if $X_1, \ldots, X_n$ are independent and are distributed as a standard a standard normal distribution $\mathcal{N}(0,1)$, and $f(X_1,\ldots, X_n) = n^{-1/2} \sum_{i = 1}^n X_i$, then $f(X_1,\ldots,X_n) \sim \mathcal{N}(0,1)$ and
\begin{equation}\label{eqn:GaussianDecay}
\mathbb{P}\Big( \abs{f(X_1,\ldots, X_n) - \IE f(X_1,\ldots, X_n)} \ge t \Big) \le 2 \exp \Big( -\frac{t^2}{2} \Big).
\end{equation}
Using the entropy method, it is possible to show that the estimate \eqref{eqn:GaussianDecay} remains true for any Lipschitz function $f$ (see e.\,g. \cite[Section 5]{Led01}). On the other hand, if $f$ is a polynomial of degree $2$, then the tails of $f(X_1,\ldots, X_n)$ are heavier. Indeed, the \emph{Hanson--Wright inequality} states that for a quadratic form in independent, standard Gaussian random variables $X_1, \ldots, X_n$ we have
\begin{equation} \label{eqn:HansonWright}
\mathbb{P} \Big( \Big \lvert \sum_{i,j = 1}^n a_{ij} X_i X_j - \mathrm{trace}(A) \Big\rvert \ge t \Big) \le 2 \exp \Big( - \frac{1}{C} \min\Big( \frac{t^2}{\norm{A}_{\mathrm{HS}}}, \frac{t}{\norm{A}_{\mathrm{op}}} \Big) \Big).
\end{equation}
Here, $\norm{A}_{\mathrm{op}}$ is the operator norm and $\norm{A}_{\mathrm{HS}}$ the Hilbert--Schmidt norm (also called Frobenius norm) of $A$ respectively. For a proof see \cite{RV13}. Thus the tails of the quadratic form decay like $\exp(-t)$ for large $t$. There are inequalities similar to \eqref{eqn:HansonWright} for the multilinear chaos in Gaussian random variables proven in \cite{La06} (and in fact, a lower bound using the same quantities as well), and in \cite{AW15} for polynomials in sub-Gaussian random variables. However, a key component is that the individual random variables $X_i$ have a sub-Gaussian tail decay, i.\,e. $\IP(\abs{X_i} \ge t) \le c\exp(-C t^2)$ for some constants $c, C$.

In recent works \cite{BGS18}, \cite{GSS18}, \cite{GSS18b} we have studied similar concentration inequalities for bounded functions $f$ of either independent or weakly dependent random variables. There, the situation is clearly different, since the distribution of $f(X_1,\ldots,X_n)$ has a compact support, and is thus sub-Gaussian, and the challenge is to give an estimate depending on different quantities derived from $f$ and $X$. However, there are many situations of interest where boundedness does not hold, such as quadratic forms in unbounded random variables as in \eqref{eqn:HansonWright}. Here it seems reasonable to focus on certain classes of functions for which the tail behavior can directly be traced back to the tails of the random variables under consideration. Therefore, in this note we restrict ourselves to polynomial functions.

In the following results, the setup is as follows. We consider independent random variables $X_1, \ldots, X_n$ which have \emph{$\alpha$-sub-exponential tail decay}. By this we mean that there exists two constants $c, C$ and a parameter $\alpha > 0$ such that for all $i = 1,\ldots,n$ and $t \ge 0$
\begin{equation}\label{eqn:TailDecayXi}
	\IP\Big( \abs{X_i} \ge t \Big) \le c \exp\Big( - \frac{t^\alpha}{C} \Big).
\end{equation}
There are many interesting choices of random variables $X_i$ of this type, such as bounded random variables (for any $\alpha > 0$), random variables with a sub-Gaussian (for $\alpha = 2$) or sub-exponential distribution ($\alpha = 1$) such as Poisson random variables, or ``fatter'' tails such as Weibull random variables with shape parameter $\alpha \in (0,1]$.

We reformulate condition \eqref{eqn:TailDecayXi} in terms of so-called \emph{exponential Orlicz norms}, but we emphasize that these two concepts are equivalent. For any random variable $X$ on a probability space $(\Omega, \mathcal{A}, \mathbb{P})$ and $\alpha > 0$ define the (quasi-)norm
\begin{equation}\label{eqn:definitionOrliczNorm}
\lVert X \rVert_{\psi_\alpha} \coloneqq \inf \left\{t > 0 \colon \mathbb{E} \exp \left(\frac{|X|^{\alpha}}{t^\alpha}\right) \le 2 \right\},
\end{equation}
adhering to the standard definition $\inf \emptyset = \infty$. Strictly speaking, this is a norm for $\alpha \ge 1$ only, since otherwise the triangle inequality does not hold. Nevertheless, the above expression makes sense for any $\alpha > 0$, and we choose to call it a norm in these cases as well. For some properties of the Orlicz norms in the case $\alpha \in (0,1]$, see Appendix \ref{section:OrliczNorms}. In this note we concentrate on values $\alpha = 2/q$ for some $q \in \mathbb{N}$, but also prove results for the case $\alpha \in (0,1]$. Throughout this work, we denote by $C$ an absolute constant and by $C_{l_1,\ldots,l_k}$ a constant that only depends on some parameters $l_1, \ldots, l_k$.

For illustration, we start with a simplified version of some of our results which may already be sufficient for application purposes. The first result is a concentration inequality which may be considered as a generalization of the Hanson--Wright inequality \eqref{eqn:HansonWright} to quadratic forms in random variables with $\alpha$-sub-exponential tail decay.

\begin{proposition}\label{proposition:weakFormHansonWright}
Let $X_1, \ldots, X_n$ be independent random variables satisfying $\IE X_i = 0, \IE X_i^2 = \sigma_i^2$, $\norm{X_i}_{\Psi_{\alpha}} \le M$ for some $\alpha \in (0,1] \cup \{2 \}$, and $A$ be a symmetric $n \times n$ matrix. For any $t > 0$ we have
\begin{equation*}
\mathbb{P} \Big( \big\lvert \sum_{i,j} a_{ij} X_i X_j - \sum_{i = 1}^n \sigma_i^2 a_{ii} \big\rvert \ge t \Big) \le 2 \exp \Big( - \frac{1}{C} \min \Big( \frac{t^2}{M^4 \norm{A}_{\mathrm{HS}}^2}, \Big( \frac{t}{M^2\norm{A}_{\mathrm{op}}} \Big)^{\frac{\alpha}{2}} \Big) \Big).
\end{equation*}
\end{proposition}

As we will see in Proposition \ref{proposition:QuadraticFormWithDiagonal}, the tail decay $\exp(-t^{\alpha/2} \norm{A}_{\mathrm{op}}^{-\alpha/2})$ (for large $t$) can be sharpened by replacing the operator norm by a smaller norm. Actually, the technical result contains up to four different regimes instead of two as above.

The next theorem provides tail estimates for polynomials in independent random variables. Note that this is not a generalization of Proposition \ref{proposition:weakFormHansonWright} due to the use of the Hilbert--Schmidt instead of the operator norms.

\begin{satz}\label{wHSn}
Let $X_1, \ldots, X_n$ be independent random variables satisfying $\norm{X_i}_{\Psi_{\alpha}} \le M$ for some $\alpha \in (0,1] \cup \{2\}$ and let $f \colon \mathbb{R}^n \to \mathbb{R}$ be a polynomial of total degree $D \in \mathbb{N}$. Then for all $t > 0$
\begin{equation}
\mathbb{P} (|f(X) - \mathbb{E}f(X)| \ge t) \le 2\exp\Big(- \frac{1}{C_{D,\alpha}} \min_{1 \le d \le D} \Big(\frac{t}{M^d \lVert \mathbb{E} f^{(d)}(X) \rVert_\mathrm{HS}} \Big)^{\alpha/d}\Big).
\end{equation}
In particular, if $\lVert \mathbb{E} f^{(d)}(X) \rVert_\mathrm{HS} \le 1$ for $d = 1, \ldots, D$, then
\[
\mathbb{E} \exp\left(\frac{C_{D,\alpha}}{M^\alpha}|f(X)|^{\frac{\alpha}{D}}\right) \le 2,
\]
or equivalently
\[
\norm{f(X)}_{\Psi_{\frac{\alpha}{D}}} \le C_{d,\alpha} M^D.
\]
\end{satz}

Intuitively, Theorem \ref{wHSn} states that a polynomial in random variables with tail decay as in \eqref{eqn:TailDecayXi} also exhibits $\alpha$-sub-exponential tail decay whenever the Hilbert--Schmidt norms are not too large. Moreover, the tail decay is ``as expected'', i.\,e. one just needs to account for the total degree $D$ by taking the $D$-th root.

One particularly interesting case is when the functional under consideration is a $d$-th order chaos. That is, given a $d$-tensor $A = (a_{i_1 \ldots i_d})$ which we assume to be symmetric, i.\,e. $a_{i_1 \ldots i_d} = a_{i_{\sigma(1)} \ldots i_{\sigma(d)}}$ for any permutation $\sigma \in \mathcal{S}_d$, we consider the polynomial
\begin{equation}		\label{eqn:chaos}
f_{d,A}(X) \coloneqq f_{d,A}(X_1, \ldots, X_n) \coloneqq \sum_{i_1,\ldots,i_d} a_{i_1 \ldots i_d} (X_{i_1}-\mathbb{E}X_{i_1}) \cdots (X_{i_d}-\mathbb{E}X_{i_d}).
\end{equation}
Additionally, we often assume that $A$ has vanishing generalized diagonal in the sense that $a_{i_1 \ldots i_d} = 0$ whenever $i_1, \ldots, i_d$ are not pairwise different. In this situation, Theorem \ref{wHSn} reads as follows:

\begin{korollar}		\label{wHSnc}
Let $X_1, \ldots, X_n$ be independent random variables with $\norm{X_i}_{\Psi_{\alpha}} \le M$ for some $\alpha \in (0,1] \cup \{2\}$ and let $A$ be a symmetric $d$-tensor with vanishing generalized diagonal such that $\lVert A \rVert_{\mathrm{HS}} \le 1$. Then 
\[
\mathbb{E} \exp\left(\frac{C_{d,\alpha}}{M^\alpha}|f_{d,A}(X)|^{\frac{\alpha}{d}} \right) \le 2.
\]
\end{korollar}

As in Theorem \ref{wHSn}, the conclusion is equivalent to a $\Psi_{\alpha/d}$-norm estimate.

\subsection{Main results}
In comparison to the aforementioned results, our main concentration inequalities provide more refined tail estimates. To this end, we need a family of tensor-product matrix norms $\lVert A \rVert_\mathcal{J}$ for a $d$-tensor $A$ and a partition $\mathcal{J} \in P_{qd}$ of $\{1, \ldots, qd \}$. For the exact definitions, we refer to \eqref{normgen}. Using these norms, we may formulate our first result for chaos-type functionals. Note that we focus on the case $\alpha = 2/q$ for some $q \in \IN$ only, which is sufficient for many applications, like products or powers of sub-Gaussian or sub-exponential random variables. The general case $\alpha \in (0,1]$ will be treated later.

\begin{satz}		\label{thmch}
Let $X_1, \ldots, X_n$ be a set of independent random variables satisfying $\lVert X_i \rVert_{\psi_{2/q}}\linebreak[3] \le M$ for some $q \in \mathbb{N}$ and $M > 0$, and let $A$ be a symmetric $d$-tensor with vanishing diagonal. Consider $f_{d,A}(X)$ as in \eqref{eqn:chaos}. Then, for any $t > 0$,
\begin{equation}	\label{eqn:Theorem11ConcentrationInequality}
\mathbb{P} (|f_{d,A}(X)| \ge t) \le 2 \exp\Big(- \frac{1}{C_{d,q}} \min_{\mathcal{J} \in P_{qd}} \Big(\frac{t}{M^d \lVert A \rVert_\mathcal{J}} \Big)^{2/|\mathcal{J}|}\Big).
\end{equation}
\end{satz}

To give an elementary example, consider the case $d=1$ and $q=2$. Here, $A = a = (a_1, \ldots, a_n)$ is a vector, and $f_{1,A}(X) = \sum_{i=1}^{n} a_i(X_i - \IE X_i)$ is just a linear functional of random variables with sub-exponential tails ($\lVert X_i \rVert_{\psi_1} \le M$). It easily follows from the definition that $\lVert A \rVert_{\{1,2\}} = |a|$ (i.\,e. the Euclidean norm of $a$) and $\lVert A \rVert_{\{\{1\},\{2\}\}} = \max_i |a_i|$. As a consequence, for any $t > 0$
$$\mathbb{P}\Big( \big\lvert \sum_{i = 1}^n a_i (X_i - \IE X_i) \big\rvert \ge t\Big) \le 2 \exp\Big(- \frac{1}{C} \min\Big(\frac{t^2}{M^2|a|^2}, \frac{t}{M\max_i|a_i|}\Big)\Big).$$
Hence, up to constants, we get back a classical result for the tails of a linear form in random variables with sub-exponential tails. For more general functions $f$ and similar results under a Poincar{\'e}-type inequality, we refer to \cite{BL97} (the first order case) and \cite{GS18} (the higher order case).

Moreover, Theorem \ref{thmch} can be used to give Hanson--Wright-type bounds 
for quadratic forms in sub-exponential random variables. Here we provide a sharpened version of Proposition \ref{proposition:weakFormHansonWright}.
Let $\skal{x,y}$ be the standard scalar product in $\mathbb{R}^n$.

\begin{proposition}		\label{proposition:QuadraticFormWithDiagonal}
Let $q \in \IN$, $A = (a_{ij})$ be a symmetric $n \times n$ matrix and let $X_1, \ldots, X_n$ be a set of independent, centered random variables with $\norm{X_i}_{\Psi_{2/q}} \le M$ and $\IE X_i^2 = \sigma_i^2$. For any $t > 0$
\begin{align*}
\mathbb{P}\Big(\big \lvert \sum_{i,j} a_{ij} X_i X_j - \sum_{i = 1}^n \sigma_i^2 a_{ii} \big \rvert \ge t \Big) \le 2 \exp \Big( - \frac{1}{C} \eta(A,q,t/M^2) \Big),
\end{align*}
where
{\scriptsize
\[
\eta(A,q,t) \coloneqq \min \left( \frac{t^2}{\norm{A}^2_{\mathrm{HS}}}, \frac{t}{\norm{A}_{\mathrm{op}}}, \Big( \frac{t}{\max_{i = 1,\ldots,n} \norm{(a_{ij})_j}_2} \Big)^{\frac{2}{q+1}}, \left( \frac{t}{\norm{A}_\infty} \right)^{\frac{1}{q}} \right).
\]}
Consequently, for any $x > 0$ we have with probability at least $1 - 2\exp(-x/C)$
{\scriptsize
\begin{equation*}
\abs{\skal{X,AX} - \IE \skal{X,AX}} \le M^2 \max\left( \sqrt{x}\norm{A}_{\mathrm{HS}}, x \norm{A}_{\mathrm{op}}, x^{\frac{q+1}{2}} \max_{i =1,\ldots,n} \norm{(a_{ij})_j}_2, x^{q} \norm{A}_\infty \right).
\end{equation*}
}
\end{proposition}

It is possible to replace $2/q$ by a general $\alpha \in (0,1] \cup \{ 2 \}$ (see Section \ref{section:alpha}). In this case, we have to replace $2/(q+1)$ by $2\alpha/(2+\alpha)$ and $1/q$ by $\alpha/2$.

\begin{bemerkung*}
Note that in comparison to the Hanson--Wright inequality \eqref{eqn:HansonWright} and Proposition \ref{proposition:weakFormHansonWright}, the more refined version contains two additional terms. The respective norms $\max_{i = 1,\ldots,n} \norm{(a_{ij})_j}_2$ and $\norm{A}_\infty$ can no longer be written in terms of the eigenvalues of $A$ (in contrast to $\norm{A}_{\mathrm{HS}}$ and $\norm{A}_{\mathrm{op}}$). Indeed, as we see later, we have $\max_{i = 1, \ldots, n} \norm{(a_{ij})_j}_2 = \norm{A}_{2 \to \infty}$, and $\norm{A}_\infty = \max_{i,j} \abs{\skal{e_i, Ae_j}}$ for the standard basis $(e_i)_i$ of $\mathbb{R}^n$. Moreover, the norms might have a very different scaling in $n$. For example, if $e = (1,\ldots,1)$ and $A = ee^T - \mathrm{Id}$, then $\norm{A}_{\mathrm{HS}} \sim \norm{A}_{\mathrm{op}} \sim n$, $\max_i \norm{(a_{ij})_j}_2 \sim n^{1/2}$ and $\norm{A}_\infty = 1$.
\end{bemerkung*}

Finally, let us state the result for general polynomials in random variables with bounded Orlicz norms. To fix some notation, if $f \colon \mathbb{R}^n \to \mathbb{R}$ is a function in $\mathcal{C}^D(\mathbb{R}^n)$, for $d \le D$ we denote by $f^{(d)}$ the (symmetric) $d$-tensor of its $d$-th order partial derivatives.

\begin{satz}\label{thmpol}
Let $X_1, \ldots, X_n$ be a set of independent random variables satisfying $\lVert X_i \rVert_{\psi_{2/q}}\linebreak[3] \le M$ for some $q \in \mathbb{N}$ and $M > 0$. Let $f \colon \mathbb{R}^n \to \mathbb{R}$ be a polynomial of total degree $D \in \mathbb{N}$. Then, for any $t > 0$,
$$\mathbb{P} (|f(X) - \mathbb{E}f(X)| \ge t) \le 2 \exp\Big(- \frac{1}{C_{D,q}} \min_{1 \le d \le D}\min_{\mathcal{J} \in P_{qd}} \Big(\frac{t}{M^d \lVert \mathbb{E} f^{(d)}(X) \rVert_\mathcal{J}} \Big)^{\frac{2}{\abs{\mathcal{J}}}} \Big).$$
\end{satz}

Note that if $f(X) = f_{D,A}(X)$ as in \eqref{eqn:chaos}, only the $D$-th order tensor gives a contribution, i.\,e. we retrieve Theorem \ref{thmch}. We discuss Theorems \ref{thmch} and \ref{thmpol} and compare them to known results in Subsection \ref{section:relatedwork}. A variant of Theorem \ref{thmpol} for polynomials in independent random variables with $\norm{X_i}_{\psi_{\alpha}} \le 1$ for any $\alpha \in (0,1]$ with be derived in Section \ref{section:alpha}.

\begin{bemerkung*}
With the help of these inequalities, it is possible to prove many results on concentration of linear and quadratic forms in independent random variables scattered throughout the literature. For example, \cite[Lemma A.6]{NSU17} is an immediate consequence of Theorem \ref{thmch} (combined with Lemma \ref{lemma:HoelderTypeInequality} for $f(X,X') = \sum_{i = 1}^n a_i X_i X_i'$). In a similar way, one can deduce \cite[Lemma C.4]{YangEtAl17} by applying Theorem \ref{thmch} to the random variable $Z_i \coloneqq X_iY_i$, whenever $(X_i, Y_i)$ is a vector with sub-exponential marginal distributions. More generally, one can consider a linear form (or higher order polynomial chaoses) in a product of $k$ random variables $X_1,\ldots, X_k$ with sub-exponential tails, for which Lemma \ref{lemma:HoelderTypeInequality} provides estimates for the $\Psi_{\frac{1}{k}}$ norm.

Lastly, the results in \cite[Appendix B]{EYY12} can be sharpened for $\alpha \in (0,1] \cup \{2 \}$ by a more general version of Proposition \ref{proposition:QuadraticFormWithDiagonal}, using the same arguments as in \cite[Section 3]{RV13} to treat complex-valued matrices.
\end{bemerkung*}

\subsection{Related work}\label{section:relatedwork}
Inequalities for the $L^p$-norms of polynomial chaos have been established in various works. From these $L^p$ norm inequalities one can quite easily derive concentration inequalities. For a thorough discussion on inequalities involving linear forms in independent random variables we refer to \cite[Chapter 1]{PG99}.

Starting with linear forms, there have been generalizations to certain classes of random variables as well as multilinear forms of higher degree (also called polynomial chaoses). Among these are the two classes of random variables with either log-convex or log-concave tails (i.\,e. $t \mapsto -\log \mathbb{P}(\abs{X} \ge t)$ is convex respectively concave). Two-sided $L^p$ norm estimates for the log-convex case were derived in \cite{HMO97} for linear forms and in \cite{KL15} for chaoses of all orders. On the other hand, for measures with log-concave tails similar two-sided estimates have been derived in \cite{GK95, La96, La99, LL03, AL12} under different conditions. Moreover, two-sided estimates for non-negative random variables have been derived in \cite{Me16} and for chaos of order two in symmetric random variables satisfying the inequality $\norm{X}_{2p} \le A \norm{X}_p$ in \cite{Me17}. 

Our approach is closer to the work of Adamczak and Wolff, \cite{AW15}, where the case of polynomials in sub-Gaussian random variables has been treated. Lastly, let us mention the two results \cite[Lemma B.2, Lemma B.3]{EYY12} and \cite[Corollary 1.6]{VW15}, where concentration inequalities for quadratic forms in independent random variables with $\alpha$-sub-exponential tails have been proven.

To be able to compare our results to the results listed above, let us discuss their conditions. Firstly, the conditions of a bounded Orlicz norm and log-convex or log-concave tails cannot be compared in general. It is known that random variables with log-convex tails satisfy $\norm{X}_{\Psi_1} < \infty$. On the other hand, the tail function of any discrete random variable $X$ is a step function (for example, if $X$ has the geometric distribution, then $-\log \IP(X \ge t) = \lfloor x \rfloor \log(1/(1-p))$), which is neither log-convex nor log-concave but can still have a finite $\Psi_{\alpha}$ norm for some $\alpha$. For example, a Poisson-distributed random variable $X$ satisfies $\norm{X}_{\Psi_1} < \infty$. 

The condition $\norm{X}_{2p} \le \alpha \norm{X}_p$ for all $p \ge 1$ and some $\alpha > 1$ used in the works of Meller implies the existence of the $\Psi_{\tilde{\alpha}}$-norm for $\tilde{\alpha} \coloneqq (\log_2\alpha)^{-1}$.
Especially in the case $\alpha = 2^d$ this yields the existence of the $\Psi_{1/d}$ norm. However, we want to stress that the results in \cite{AL12, KL15, Me16, Me17} are two-sided and require very different tools.

Moreover, the two works of Schudy and Sviridenko \cite{SS12b, SS12a} contain concentration inequalities for polynomials in so-called \emph{moment bounded random} variables. Therein, a random variable $Z$ is called moment bounded with parameter $L > 0$, if for all $i \ge 1$
$
\IE \abs{Z}^i \le i L \IE \abs{Z}^{i-1}.
$
Actually, using Stirling's formula, it is easy to see that moment-boundedness implies $\norm{Z}_{\Psi_1} < \infty$, but it is not clear whether the converse implication also holds. However, there is no inequality of the form $L \le C \norm{X}_{\Psi_1}$, as can be seen by $X \sim \mathrm{Ber}(p)$.

Considering quadratic forms in random variables $X$ which are moment bounded and centered, one can easily see that (apart from the constants) the bound in Proposition \ref{proposition:QuadraticFormWithDiagonal} is sharper than the corresponding inequality in \cite[Theorem 1.1]{SS12a}. Since for log-convex distributions there are two-sided estimates, Proposition \ref{proposition:QuadraticFormWithDiagonal} is sharp in this class. Apart from quadratic forms, due to the different conditions and quantities, it is difficult to compare \cite{SS12a} and Theorem \ref{thmpol} in general.

\subsection{Outline}
In Section \ref{section:applications} we formulate and prove several applications which can be deduced from the main results. Section \ref{section:ProofThm13} contains the proof for the concentration inequalities for multilinear forms (Theorem \ref{thmch}). Thereafter, we provide the proof of Proposition \ref{proposition:QuadraticFormWithDiagonal} in Section \ref{section:QuadraticForms} and of Theorem \ref{thmpol} in Section \ref{section:ProofThm16}. Section \ref{section:alpha} is devoted to some extensions of the main results for random variables with finite Orlicz-norms for any $\alpha \in (0,1]$. Lastly, we finish this note by collecting some elementary properties of the Orlicz-norms in the Appendix \ref{section:OrliczNorms}.

\section{Applications}\label{section:applications}
In the following, we provide some applications of our main results. In particular, all the results in this section follow from either Proposition \ref{proposition:weakFormHansonWright} or \ref{proposition:QuadraticFormWithDiagonal}. For any random variables $X_1, \ldots, X_n$ we write $X = (X_1, \ldots, X_n)$.

\subsection{Concentration of the Euclidean norm of a vector with independent components}
As a start, Proposition \ref{proposition:weakFormHansonWright} can be used to give concentration properties of the Euclidean norm of a linear transformation of $X$ consisting of independent, normalized random variables with sub-exponential tails. We give two different forms thereof. The first form is inspired by the results in \cite{RV13} for sub-Gaussian random variables.

\begin{proposition}\label{proposition:EuclideanNormVector}
Let $X_1, \ldots, X_n$ be independent random variables satisfying $\IE X_i = 0, \IE X_i^2 = 1, \norm{X_i}_{\Psi_\alpha} \le M$ for some $\alpha \in (0,1] \cup \{2 \}$ and let $B \neq 0$ be an $m \times n$ matrix. For any $c > 0$ and any $t \ge c \norm{B}_{\mathrm{HS}}$ we have
\begin{equation}
	\IP \Big( \abs{\norm{BX}_2 - \norm{B}_{\mathrm{HS}}} \ge t \Big) \le 2 \exp \Big( - \frac{\min(c^{2-\alpha},1)}{C M^4 \norm{B}_{\mathrm{op}}^\alpha} t^\alpha \Big).
\end{equation}
\end{proposition}

Note that in the case $\alpha = 2$ the constant is not present on the right hand side and thus we can choose any $t > 0$, which is exactly \cite[Theorem 2.1]{RV13}. In the general case, we need to restrict $t$ to be of the order $\norm{B}_{\mathrm{HS}}$. 

The assumption of unit variance can be weakened, with some minor modifications, i.\,e. $\norm{B}_{\mathrm{HS}}$ has to be replaced by $(\sum_{i = 1}^n \sigma_i^2 \sum_{j = 1}^n b_{ij}^2)^{1/2}$ and the constant $C$ will depend on $\min_{i =1,\ldots,n} \sigma_i^2$. We omit the details.

\begin{proof}
First off, note that it suffices to prove the inequality for a matrix $B$ such that $\norm{B}_{\mathrm{HS}} = 1$ and $t \ge c$, since the general case follows by considering $\tilde{B} \coloneqq B \norm{B}_{\mathrm{HS}}^{-1}$.

Let us apply Proposition \ref{proposition:weakFormHansonWright} to the matrix $A \coloneqq B^T B$. An easy calculation shows that $\mathrm{trace}(A) = \mathrm{trace}(B^T B) = \norm{B}_{\mathrm{HS}}^2 = 1$, so that we have
\begin{align}\label{eqn:eqn1}
\begin{split}
	\IP\Big( \abs{ \norm{BX}_2^2 - 1} \ge t \Big) &\le 2 \exp \Big( - \frac{1}{C M^4} \min \Big( \frac{t^2}{\norm{B}_{\mathrm{op}}^2}, \Big( \frac{t}{\norm{B}_{\mathrm{op}}^2}\Big)^{\frac{\alpha}{2}}\Big) \Big) \\
	&= 2 \exp \Big( - \frac{1}{CM^4 \norm{B}_{\mathrm{op}}^\alpha} \min \Big( \frac{t^{2-\alpha}}{\norm{B}^{2-\alpha}_{\mathrm{op}}} t^\alpha, t^{\frac \alpha 2} \Big)\Big) \\
	&\le 2 \exp \Big( -\frac{\min(c^{2-\alpha} t^\alpha, t^{\frac \alpha 2})}{CM^4 \norm{B}_{\mathrm{op}}^\alpha} \Big) \\
	&\le 2 \exp \Big( - \frac{\min(c^{2-\alpha},1)}{CM^4 \norm{B}_{\mathrm{op}}^\alpha} \min(t^\alpha, t^{\frac{\alpha}{2}}) \Big).
\end{split}
\end{align}
Here, in the first step we have used the estimates  $\norm{A}_{\mathrm{HS}}^2 \le \norm{B}_{\mathrm{op}}^2 \norm{B}_{\mathrm{HS}}^2 = \norm{B}_{\mathrm{op}}^2$ and $\norm{A}_{\mathrm{op}} \le \norm{B}_{\mathrm{op}}^2$ as well as the fact that by Lemma \ref{lemma:PsiAlphaAndGrowth}, $\mathbb{E} X_i^2=1$ for any $i$ implies $M \ge C_\alpha > 0$. The second inequality follows from $t \ge c \ge c \norm{B}_{\mathrm{op}}$ and the third inequality is a consequence of $\min(c^{2-\alpha}t^\alpha, t^{\frac \alpha 2}) \ge \min(c^{2-\alpha},1) \min(t^\alpha, t^{\frac \alpha 2}).$

Now, as in \cite{RV13}, we use the inequality $\abs{z - 1} \le \min( \abs{z^2 -1}, \abs{z^2 - 1}^{1/2})$, giving for any $t > 0$
\begin{equation}\label{eqn:eqn2}
\IP \Big( \abs{\norm{BX}_2 - 1} \ge t \Big) \le \IP \Big( \abs{\norm{BX}_2^2 - 1} \ge \max(t,t^2) \Big).
\end{equation}
Hence, a combination of \eqref{eqn:eqn1}, \eqref{eqn:eqn2} and $\min(\max(r,r^2), \max(r^{1/2},r)) = r$ yields for $t > c$
\begin{equation}
\IP \Big( \abs{\norm{BX}_2 - 1} \ge t \Big) \le 2 \exp\Big( - \frac{\min(c^{2-\alpha},1)}{CM^4 \norm{B}_{\mathrm{op}}^\alpha } t^\alpha \Big).
\end{equation}
\end{proof}

The next corollary provides an alternative estimate for $\norm{BX}_2$:

\begin{korollar}\label{corollary:ConcentrationNorm}
Let $X_1, \ldots, X_n$ be independent, centered random variables satisfying $\norm{X_i}_{\Psi_1} \le M$ and $\IE X_i^2 = \sigma_i^2$. For an $n \times n$ matrix $B$ with real entries let $A = B^T B = (a_{ij})$. Then, for any $x > 0$, with probability at least $1 - 2\exp(- x/C)$ we have
{\footnotesize
\[
\norm{BX}_2^2 \le \sum_{i = 1}^n \sigma_i^2 \sum_{j = 1}^n b_{ji}^2 + M^2 \max \Big( \sqrt{x} \norm{A}_{\mathrm{HS}}, x \norm{A}_{\mathrm{op}}, x^{3/2} \max_{i = 1,\ldots,n} \norm{(a_{ij})_j}_2, x^2 \norm{A}_\infty \Big).
\]
}
\end{korollar}

Corollary \ref{corollary:ConcentrationNorm} can be compared to various bounds on the norms of $\norm{BX}_2$ in the case that $X$ is a sub-Gaussian vector (see for example \cite{HKZ12} or \cite{Ad15}). For sub-Gaussian vectors with sub-Gaussian constant $1$, we have with probability at least $1 - \exp(-x)$
\[
  \norm{BX}_2^2 \le \mathrm{trace}(B^T B) + 2 \norm{B^T B}_{\mathrm{HS}} \sqrt{x} + 2 \norm{B^T B}_{\mathrm{op}} x,
\]
so that we have similar terms corresponding to $\sqrt{x}$ and $x$, whereas in the sub-exponential case we need two additional terms to account for the heavier tails of its components.

\begin{proof}
Define the quadratic form
\[
Z \coloneqq \norm{BX}_2^2 = \skal{BX,BX} = \skal{X,B^TB X} = \skal{X,AX}.
\]
Using Proposition \ref{proposition:QuadraticFormWithDiagonal} with the matrix $A$ gives with probability $1 - 2 \exp(- x/C)$ 
\begin{align*}
\abs{Z - \IE Z} \le \max \Big( \sqrt{x} \norm{A}_{\mathrm{HS}}, x \norm{A}_{\mathrm{op}}, x^{3/2} \max_{i = 1,\ldots,n} \norm{A_{i\cdot}}_2, x^2 \norm{A}_\infty \Big).
\end{align*}
From these inequalities and $\abs{x} \ge x$ the claim easily follows by taking the square root. Note that $\IE Z = \IE \skal{X,AX} = \sum_{i = 1}^n \sigma_i^2 \sum_{j = 1}^n b_{ji}^2$.
\end{proof}

\subsection{Projections of a random vector and distance to a fixed subspace}
It is possible to apply Proposition \ref{proposition:QuadraticFormWithDiagonal} to any matrix $A$  associated to an orthogonal projection. In these cases, the norms can be explicitly calculated. Moreover, these norms do not depend on the structure of the subspace onto which one projects, but merely on its dimension. This leads to the following application, where we replace a fixed projection by a random one.

\begin{korollar}\label{corollary:randomProjection}
	Let $X_1, \ldots, X_n$ be independent random variables satisfying $\IE X_i = 0, \IE X_i^2 = \sigma_i^2$ and $\norm{X_i}_{\Psi_1} \le M$. Furthermore, let $m < n$ and $P$ be the (random) orthogonal projection onto an $m$-dimensional subspace of $\mathbb{R}^n$, distributed according to the Haar measure on the Grassmanian manifold $G_{m,n}$. For any $x > 0$, with probability at least $1 - 2 \exp (- x/C)$, we have
	\begin{equation}
		\Big\lvert \norm{PX}_2^2 - \frac{m}{n} \sum_{i = 1}^n \sigma_i^2 \Big\rvert \le M^2 \max \Big( \sqrt{xm}, x^2 \Big).
	\end{equation}
\end{korollar}

\begin{proof}[Proof of Corollary \ref{corollary:randomProjection}]
	This is an application of Proposition \ref{proposition:QuadraticFormWithDiagonal}. To see
	\[
		\IE \norm{PX}^2_2 = \frac{m}{n} \sum_{i = 1}^n \sigma_i^2,
	\] 
	we use \cite[Lemma 5.3.2]{Ver18} conditionally on $X$, i.\,e. we have
	\[
	\IE \norm{PX}_2^2 = \IE \IE \left( \norm{PX}_2^2 \mid X \right) = \frac{m}{n} \IE \norm{X}_2^2 = \frac{m}{n} \sum_{i = 1}^n \sigma_i^2.
	\]
	Moreover, for any projection $P$ onto an $m$-dimensional subspace, one can see that $\norm{P}_{\mathrm{HS}}^2 = \sum_{i = 1}^n \lambda_i(P)^2 = m$. Moreover, it is clear that $\norm{P}_\infty \le \max_{i = 1,\ldots,n} \norm{(p_{ij})_j}_2 \le \norm{P}_{2 \to 2} = 1$.
\end{proof}

A very similar result which follows from Proposition \ref{proposition:EuclideanNormVector} is the following variant of \cite[Corollary 3.1]{RV13}. We use the notation $d(X,E) = \inf_{e \in E} d(X,e)$ for the distance between an element $X$ and a subset $E$ of a metric space $(M,d)$.

\begin{korollar}\label{corollary:distanceToSubspace}
Let $X_1, \ldots, X_n$ be independent random variables satisfying $\IE X_i = 0, \IE X_i^2 = 1$ and $\norm{X_i}_{\Psi_\alpha} \le M$ for some $\alpha \in (0,1] \cup \{2 \}$, and let $E$ be a subspace of $\mathbb{R}^n$ of dimension $d$. For any $t \ge \sqrt{n-d}$ we have
\[
\mathbb{P} \Big( \abs{d(X,E) - \sqrt{n-d}} \ge t \Big) \le 2 \exp \Big( - \frac{t^\alpha}{ CM^4} \Big).
\]
\end{korollar}

\begin{proof}
This follows exactly as in \cite[Corollary 3.1]{RV13} by using Proposition \ref{proposition:EuclideanNormVector}.
\end{proof}

\subsection{Spectral bound for a product of a fixed and a random matrix}
We can also extend the second application in \cite{RV13} to any $\alpha$-sub-exponential random vector as follows.

\begin{proposition}\label{proposition:randomMatricesRV13}
Let $B$ be a fixed $m \times N$ matrix and let $G$ be a $N \times n$ random matrix with independent entries satisfying $\IE g_{ij} = 0, \IE g_{ij}^2 = 1$ and $\norm{g_{ij}}_{\Psi_\alpha} \le M$ for some $\alpha \in (0,1]$. For any $u, v \ge 1$ with probability at least $1 - 2 \exp(- u^\alpha r(B)^\alpha - v^\alpha n)$ we have
\[
\norm{BG}_{\mathrm{op}} \le 4 C_\alpha M^{4/\alpha} \Big( u \norm{B}_{\mathrm{HS}} + v n^{1/\alpha} \norm{B}_{\mathrm{op}} \Big).
\]
\end{proposition}

\begin{proof}
We mimic the proof of \cite[Theorem 3.2]{RV13}. For any fixed $x \in S^{n-1}$ consider the linear operator $T: \mathbb{R}^{Nn} \to \mathbb{R}^m$ given by $T(G) = BGx$, and (by abuse of notation) write $T$ for the matrix corresponding to this linear map in the standard basis. Using Proposition \ref{proposition:EuclideanNormVector} applied to the matrix $T$ we have
\[
\IP\Big( \abs{\norm{BGx}_2 - \norm{T}_{\mathrm{HS}}} \ge t \Big) \le 2 \exp \Big( - \frac{t^\alpha}{CM^4 \norm{T}_{\mathrm{op}}^\alpha} \Big).
\]
Now, since $\norm{T}_{\mathrm{HS}} = \norm{B}_{\mathrm{HS}}$ and $\norm{T}_{\mathrm{op}} \le \norm{B}_{\mathrm{op}}$, this yields for any $t \ge \norm{B}_{\mathrm{HS}}$
\[
\IP \Big( \norm{BGx}_2 > \norm{B}_{\mathrm{HS}} + t \Big) \le 2 \exp \Big(- \frac{t^\alpha}{CM^4 \norm{B}_{\mathrm{op}}^\alpha} \Big).
\]
If we define $t = (2 CM^4)^{1/\alpha}\Big( u \norm{B}_{\mathrm{HS}} + (\log(5)+1)^{1/\alpha} v n^{1/\alpha} \norm{B}_{\mathrm{op}} \Big)$ for arbitrary $u,v \ge 1$ and use the inequality $2(r+s)^\alpha \ge (r^\alpha + s^\alpha)$ valid for all $r, s \ge 0$, we obtain
\begin{align*}
\IP \Big( \norm{BGx}_2 > \norm{B}_{\mathrm{HS}} + t \Big) &\le 2 \exp \Big( - \frac{u^\alpha \norm{B}_{\mathrm{HS}}^\alpha + v^\alpha n(\log(5)+1) \norm{B}_{\mathrm{op}}^\alpha}{\norm{B}_{\mathrm{op}}^\alpha} \Big) \\
&\le 2 \exp \Big( - u^\alpha r(B)^\alpha - v^\alpha n - v^\alpha n \log(5) \Big) \\
&\le 5^{-n} 2 \exp\Big( - u^\alpha r(B)^\alpha - v^\alpha n \Big).
\end{align*}
The last step is again a covering argument as in \cite{RV13}. Choose a $1/2$-covering $\mathcal{N}$ (satisfying $\abs{\mathcal{N}} \le 5^n$, see \cite[Lemma 5.2]{Ver12}) of the unit sphere in $\mathbb{R}^n$, and note that a union bound gives
\begin{align*}
\mathbb{P}\Big( \bigcap_{x \in \mathcal{N}} \norm{BGx}_2 \le \norm{B}_{\mathrm{HS}} + t \Big) &\ge 1 - \sum_{x \in \mathcal{N}} \IP\Big( \norm{BGx}_2 > \norm{B}_{\mathrm{HS}} + t \Big) \\
&\ge 1 - 2 \exp \Big( -u^\alpha r(B)^\alpha - v^\alpha n \Big).
\end{align*}
Lemma 5.3 in \cite[Lemma 5.3]{Ver12} yields
\[
	\norm{BG}_{\mathrm{op}} \le 2 \max_{x \in \mathcal{N}} \norm{BGx}_2 \le 2(\norm{B}_{\mathrm{HS}} + t),
\] 
from which the assertion easily follows by upper bounding and simplifying the expression $2 \norm{B}_{\mathrm{{HS}}} + 2t$.
\end{proof}

\subsection{Special cases} 
It is possible to apply all results to random variables having a Poisson distribution, i.\,e. $X_i \sim \mathrm{Poi}(\lambda_i)$ for some $\lambda_i \in (0,\infty)$. By using the moment generating function of the Poisson distribution, it is easily seen that
\[
\norm{X_i}_{\Psi_1} = \frac{1}{\log\Big( \log(2) \lambda_i^{-1} + 1 \Big)} \eqqcolon g(\lambda_i).
\]
The function $g$ is increasing and satisfies $g(x) \sim \log(1/x)$ (for $x \to 0$) and $g(x) \sim x/\log(2)$ (for $x \to \infty$).
More generally, if the random variable $\abs{X}$ has a moment generating function $\phi_{\abs{X}}$ in a neighborhood of $0$, it can be used to explicitly calculate the $\Psi_1$-norm. Indeed, we have $\IE \exp(\abs{X}/t) = \phi_{\abs{X}}(t^{-1})$, and so $\norm{X}_{\Psi_1} = 1/\phi^{-1}_{\abs{X}}(2)$.

Thus, as a special case of Proposition \ref{proposition:QuadraticFormWithDiagonal}, we obtain the following corollary.

\begin{korollar}
Let $X_i \sim \mathrm{Poi}(\lambda_i)$, $B \coloneqq g(\max_{i = 1,\ldots,n} \lambda_i)$ and $A = (a_{ij})$ be a symmetric $n \times n$ matrix. We have for any $t > 0$
\begin{align*}
&\IP\Big( \Big\lvert \sum_{i,j} a_{ij} X_i X_j - \sum_{i = 1}^n a_{ii} \lambda_i \Big\rvert \ge B^2 t \Big) \\
&\le 2 \exp \left( - \frac{1}{C} \min \left( \frac{t^2}{\norm{A}^2_{\mathrm{HS}}}, \frac{t}{\norm{A}_{\mathrm{op}}}, \Big( \frac{t}{\max_{i} \norm{(a_{ij})_j}_2} \Big)^{\frac{2}{3}}, \left( \frac{t}{\norm{A}_\infty} \right)^{\frac{1}{2}} \right) \right) \\
&\le 2 \exp \left(- \frac{1}{C} \min\left( \frac{t^2}{\norm{A}^2_{\mathrm{HS}}}, \left( \frac{t}{\norm{A}_{\mathrm{op}}} \right)^{\frac{1}{2}} \right) \right).
\end{align*}
\end{korollar}

For Poisson chaos of arbitrary order $d \in \mathbb{N}$, one may derive similar results by evaluating Theorem \ref{thmch} or Corollary \ref{corollary:PsiAlphaKL15} (both for $\alpha=1$). Note though that already for $d=1$, we lose a logarithmic factor in the exponent. However, we are not aware of any more refined fluctuation estimates for $d \ge 2$.

Another interesting example of a sub-exponential random variable arises in stochastic geometry. If $K \subseteq \mathbb{R}^n$ is an isotropic, convex body and $X$ is distributed according to the cone measure on $K$, then $\norm{\skal{X,\theta}}_{\Psi_1} \le c$ for some constant $c$ and any $\theta \in S^{n-1}$. For the details and the proof we refer to \cite[Lemma 5.1]{PTT18}.

\subsection{Concentration properties for fixed design linear regression}
It is possible to extend the example of the fixed design linear regression in \cite{HKZ12} to the situation of a sub-exponential noise (instead of sub-Gaussian).

To this end, let $y_1, \ldots, y_n \in \mathbb{R}^d$ be fixed vectors (commonly called the \emph{design vectors}), $Y = (y_1, \ldots, y_n)$ (the $d \times n$ \emph{design matrix}) and assume that the $d \times d$ matrix $\Sigma = n^{-1} \sum_{i = 1}^n y_i y_i^T$ is invertible; in this case, define $B \coloneqq n^{-1} \Sigma^{-1/2} Y \in M(d \times n)$. Let $X_1, \ldots, X_n$ be independent random variables with $\norm{X_i}_{\Psi_1} \le M$ and define
\begin{align*}
\beta &\coloneqq n^{-1} \sum_{i = 1}^n \IE X_i \Sigma^{-1}y_i \\
\hat{\beta}(X) &\coloneqq n^{-1} \sum_{i = 1}^n X_i \Sigma^{-1}y_i.
\end{align*}
$\beta$ is the coefficient vector of the least expected squared error and $\hat{\beta}(X)$ is its ordinary least squares estimator (given the observation $X$). The quality of the estimator $\hat{\beta}$ can be judged by the excess loss
\begin{equation}
R(X) = \norm{\Sigma^{1/2}(\hat{\beta}(X) - \beta)}^2 = \sum_{i,j} a_{ij} (X_i - \IE X_i)(X_j - \IE X_j), 
\end{equation}
where $A = (a_{ij}) = B^T B = n^{-2} Y^T \Sigma^{-1} Y$, as can be shown by elementary calculations. Observe that this is a quadratic form in $X_i$ with coefficients depending on the vectors $y_i$. Thus, Proposition \ref{proposition:QuadraticFormWithDiagonal} yields the following corollary.

\begin{korollar}
In the above setting, for any $x > 0$ the inequality
\[
  \abs{R(X) - \IE R(X)} \le 4M^2 \max \left( \sqrt{x} \norm{A}_{\mathrm{HS}}, x \norm{A}_{\mathrm{op}} , x^{3/2} \max_{i=1,\ldots,n} \norm{(a_{ij})_j}_2, x^2 \norm{A}_\infty \right)
\]
holds with probability at least $1 - 2\exp(- x/C)$.
\end{korollar}

Thus the concentration properties of $R(X)$ around its mean depends on the four different norms of the matrix $A$. The factor $4$ appears due to the necessary centering of the $X_i$.

\subsection{Central limit theorems for quadratic forms and random edge weights}
In this section, our aim is to quantify central limit theorems for quadratic forms $Q(X) = Q_{A}(X) = \sum_{i,j} a_{ij} X_i X_j$ in sub-exponential random variables $X_1, \ldots, X_n$ using concentration of measure results. Typically, the first step is finding conditions such that $Q(X)$ can be approximated by a linear form $L(X)$. This reduces the problem to finding conditions such that $L(X)$ is asymptotically normal (e.\,g. using the Lyapunov central limit theorem).

The weak convergence of quadratic forms to a normal distribution is classical, and we refer to \cite{dJ87} and \cite{GT99}, \cite{Cha08} for general statements (and rates of convergence), as well as \cite{PG99} for general statements on central limit theorems for $U$-statistics.

Let us first consider the task of approximating $Q(X)$ by a linear form $L(X)$. To this end, assume that $A$ is symmetric with vanishing diagonal and $\mathbb{E}X_i \ne 0$ for some $i$. Then, we may decompose
\begin{equation*}
Q(X) = \sum_{i,j} a_{ij} (X_i-\mathbb{E}X_i)(X_j-\mathbb{E}X_j) + 2\sum_{i=1}^{n} \big(\sum_{j=1}^{n}a_{ij}\mathbb{E}X_j\big) (X_i-\mathbb{E}X_i) + \mathbb{E}Q(X)
\end{equation*}
(this is in fact the Hoeffding decomposition of $Q(X)$), and we therefore define 
\[
L(X) \coloneqq \sum_{i=1}^{n} \big(\sum_{j=1}^{n}a_{ij}\mathbb{E}X_j\big) (X_i-\mathbb{E}X_i) \eqqcolon \sum_{i=1}^{n} c_{A,i} (X_i-\mathbb{E}X_i). \] 
Obviously, $\mathrm{Var}(L(X)) = \sum_{i = 1}^n c_{A,i}^2\mathrm{Var}(X_i)$. Thus, under the condition
\begin{equation}\label{eqn:MeanFieldApproxCondition}
\lim_{n \to \infty} \frac{\sum_{i = 1}^n c_{A,i}^2 \mathrm{Var}(X_i)}{\norm{A}_{\mathrm{HS}}^2} = \infty,
\end{equation}
the asymptotic behavior of the properly normalized quadratic form is dominated by the linear term. Under additional assumptions of the tail behavior of the $X_i$, the approximation can also be quantified.

\begin{lemma}\label{proposition:CLTsubexp}
Let $X = (X_n)_{n \in \IN}$ be a sequence of independent random variables with $\norm{X_i}_{\Psi_1} \le M$ for a constant $M > 0$ and assume $\IE X = (\IE X_n)_{n \in \IN} \neq 0$, $\mathrm{Var}(X) \coloneqq (\mathrm{Var}(X_n))_{n \in \IN} \neq 0$. Furthermore, let $A = A^{(n)}$ be a sequence of symmetric matrices satisfying \eqref{eqn:MeanFieldApproxCondition}. Then for any $t > 0$
{\footnotesize
\[
\IP \left( \abs{Q(X) - \IE Q(X) - 2L(X)} \ge t \right) \le 2 \exp \left( -\min \left( \frac{t^2 \mathrm{Var}(L(X))^2}{\norm{A}_{\mathrm{HS}}^2}, \frac{t^{1/2} \mathrm{Var}(L(X))^{1/2}}{\norm{A}_{\mathrm{HS}}^{1/2}} \right) \right).
\]
}
\end{lemma}

\begin{proof}
Rewrite the Hoeffding decomposition of $Q$ as
{\footnotesize
  \begin{equation}
  \skal{A^{(n)}(X-\IE X), X - \IE X} = \skal{A^{(n)}X,X} - \IE \skal{A^{(n)}X,X} - 2 \skal{X, A \IE X} + 2 \IE \skal{X, A \IE X},
  \end{equation}}
and recall $c_n = \mathrm{Var}(L(X))$. An application of Theorem \ref{thmch} yields
  {\footnotesize
  \begin{equation}
    \IP\Big( \abs{\skal{A^{(n)}(X - \IE X), (X - \IE X)}} \ge c_n t \Big) \le 2 \exp \Big( -\frac{1}{C_d} \min \Big( \Big(\frac{c_n t}{\norm{A}_{\mathrm{HS}}}\Big)^2, \Big( \frac{c_n t}{\norm{A}_{\mathrm{HS}}} \Big)^{1/2} \Big) \Big).
  \end{equation}}
\end{proof}

In the case that the $X_i$ are also identically distributed, \eqref{eqn:MeanFieldApproxCondition} is equivalent to $\norm{A^{(n)}}_{\mathrm{HS}}^2 = o(\sum_{i = 1}^n ( \sum_{j = 1}^n A^{(n)}_{ij})^2)$. For example it is satisfied for $A^{(n)} = ee^T - \mathrm{Id}$, where $e = (1,\ldots,1)^T \in \mathbb{R}^n$.

We may apply these results to sequences of graphs. Here we always assume that the $X_i$ are identically distributed. For each $n$, let $G_n = (V_n,E_n)$ be some undirected graph on $n$ nodes (which we may consider as a kind of ``base graph''). If $A = A^{(n)}$ denotes its adjacency matrix, then \eqref{eqn:MeanFieldApproxCondition} can be rewritten as
\begin{equation}\label{eqn:conditionForGraphs}
\frac{\sum_{v \in V_n} \mathrm{deg}(v)^2}{2\abs{E_n}} \to \infty.
\end{equation}
Sequences of graphs satisfying \eqref{eqn:conditionForGraphs} are the complete graph, the complete bipartite graph $G_n = K_{m_1(n),m_2(n)}$ for parameters $m_1(n), m_2(n)$ satisfying $m_1(n) + m_2(n) \to \infty$ and $d_n$-regular graphs for $d_n \to \infty$.

The example of the $n$-stars shows that \eqref{eqn:conditionForGraphs} is not sufficient for a central limit theorem of the quadratic form. Indeed, in this case we have $Q(X) = X_1 \sum_{i = 2}^N X_i$, where $1$ is the vertex with degree $(n-1)$. As is easily seen, $Q(X) = 0$ on $\{ X_1 = 0 \}$, and thus if $X$ are Bernoulli distributed, the distribution has an atom which does not vanish for $n \to \infty$.

Finally, let us provide an example of a sequence of graphs for which a central limit theorem can be shown by imposing additional conditions. Here we assume that the random variables $X_i$ are non-negative. In this case, they can be used to define edge weights $w_n(X): E_n \to \mathbb{R}_+$ by $w_n(\{i,j\})(X) = X_i X_j$. Also let $W_n(X) \coloneqq \sum_{e \in E_n} w_n(e)(X)$ be the total edge weight. Note that $W_n(X) = \skal{AX,X}$ for the adjacency matrix $A$ of $G$.

\begin{proposition}\label{proposition:CLTinGraphs}
Let $X$ be a non-negative random variable with $\norm{X}_{\Psi_1} \le M$ and $\IE X = \lambda > 0, \mathrm{Var}(X) = \sigma^2 > 0$, and let $(X_n)_{n \in \IN}$ be a sequence of independent copies. Consider a sequence $G_n = (V_n, E_n)$ of graphs with $V_n = \{1,\ldots,n\}$ such that \eqref{eqn:conditionForGraphs} and
\begin{equation}\label{eqn:degreeToZero}
\frac{\left( \sum_{v \in V_n} \mathrm{deg}(v)^3 \right)^2}{\left( \sum_{v \in V_n} \mathrm{deg}(v)^2 \right)^3} \to 0
\end{equation}
hold. 
Then, for the total edge weight $W_n(X)$, we have
\[
\frac{W_n(X) - \IE W_n(X)}{2 \lambda \sigma \Big( \sum_{v \in V_n} \mathrm{deg}(v)^2 \Big)^{1/2}} \Rightarrow \mathcal{N}(0,1).
\]
\end{proposition}

Note that $W_n(X)$ is neither a sum of independent random variables, nor can be it written as a sum of an $m$-dependent sequence, since $w(e)(X)$ and $w(f)(X)$ are dependent whenever $e \cap f \neq \emptyset$.

In the case that $X \sim \mathrm{Ber}(p)$, the quantity $W_n(X)$ has a nice interpretation. If we interpret $X_v = 0$ as a failed vertex in the ``base graph'' $G_n$, $W_n(X)$ is the number of edges in the subgraph that is induced by the (random) vertex set $\{v \in V_n : X_{v} = 1 \}$.

\begin{proof}
Consider the linear approximation given in Lemma \ref{proposition:CLTsubexp}
\[
L(X) \coloneqq \left(\sum_i c_i(A^{(n)})^2 \right)^{-1/2} \sum_{i = 1}^n c_i(A^{(n)}) X_i.
\]
It is also easy to see that condition \eqref{eqn:degreeToZero} implies Lyapunov's condition with $\delta = 1$. Consequently, by Lindeberg's central limit theorem
\[
L(X) \Rightarrow \mathcal{N}(0,1).
\]
The claim now easily follows by combining Lemma \ref{proposition:CLTsubexp} and Slutsky's theorem.
\end{proof}

It should be possible to extend the result to any sequence of random graphs satisfying \eqref{eqn:conditionForGraphs} and \eqref{eqn:degreeToZero} by conditioning. Moreover, with appropriately modified conditions, by a more refined analysis it is possible to vary the sub-exponential constant $M$ with $n$. We omit the details.

\section{The multilinear case: Proof of Theorem \ref{thmch}}\label{section:ProofThm13}
To begin with, let us introduce some notation. Define $[n] \coloneqq \{1, \ldots, n \}$, and let $\mathbf{i} = (i_1, \ldots, i_d) \in [n]^d$ be a multiindex. For any subset $C \subseteq [d]$ with cardinality $\abs{C} > 1$, we may introduce the ``generalized diagonal'' of $[n]^d$ with respect to $C$ by
\begin{equation}\label{gendiag}
\{\mathbf{i} \in [n]^d : i_k = i_l \ \text{for all} \ k,l \in C \}.
\end{equation}
This notion of generalized diagonals naturally extends to $d$-tensors $A = (a_\mathbf{i})_{\mathbf{i} \in [n]^d}$ (obviously, the generalized diagonal of $A$ with respect to $C$ is the set of coefficients $a_\mathbf{i}$ such that $\mathbf{i}$ lies on the generalized diagonal of $[n]^d$ with respect to $C$).
If $d = 2$ and $C = \{1,2\}$, this gives back the usual notion of the diagonal of an $n \times n$ matrix. Moreover, write
$$[n]^{\underline{d}} \coloneqq \{\mathbf{i} \in [n]^d : i_1, \ldots, i_d \ \text{are pairwise different} \}.$$

If $A, B$ are $d$-tensors, we define $\langle A, B \rangle = \sum_{\mathbf{i} \in [n]^d} a_\mathbf{i} b_\mathbf{i}$. Given a set of $d$ vectors $v^1, \ldots, v^d \in \mathbb{R}^n$, we write $v^1 \otimes \ldots \otimes v^d$ for the outer product
$$(v^1 \otimes \ldots \otimes v^d)_{i_1 \ldots i_d} \coloneqq \prod_{j=1}^{d} v^j_{i_j}.$$
In fact, $v^1 \otimes \ldots \otimes v^d$ is a $d$-tensor. In particular, we may regard $A$ as a multilinear form by setting $A(v^1, \ldots, v^d) \coloneqq \langle A, v^1 \otimes \ldots \otimes v^d \rangle$ for any $v^1, \ldots, v^d \in \mathbb{R}^n$.

The latter idea may be generalized by noting that any partition $\mathcal{J} = \{J_1,\ldots, J_k\}$ of $[d]$ induces a partition of the space of $d$-tensors as follows. Identify the space of all $d$-tensors with $\mathbb{R}^{n^d}$ and decompose
\begin{equation}\label{idf}
\mathbb{R}^{n^d} \cong \bigotimes_{i = 1}^k \mathbb{R}^{n^{J_i}} \cong  \bigotimes_{i = 1}^k \bigotimes_{j \in J_i} \mathbb{R}^n.
\end{equation}
For any $x = x^{(1)} \otimes \ldots \otimes x^{(k)}$, the identification with a $d$-tensor is given by $x_{J_1,\ldots,J_d} = \prod_{l = 1}^k x^{(l)}_{J_{I_l}}$. For example, for $d = 4$ and $\mathcal{I} = \{ \{1,4\}, \{2,3\} \}$ we have two matrices $x,y$ and $x_{J_1,J_2,J_3,J_4} = x_{J_1 J_4} y_{J_2 J_3}$. Using this representation, any $d$-tensor $A$ can be trivially identified with a linear functional on $\mathbb{R}^{n^d}$ via the standard scalar product, i.\,e.
\[
	Ax = A\left(x^{(1)} \otimes \ldots \otimes x^{(k)}\right) = \skal{A, x^{(1)} \otimes \ldots \otimes x^{(k)}} = \sum_{\mathbf{i} \in [n]^d} a_{\mathbf{i}} \prod_{l = 1}^d x^{(l)}_{\mathbf{i}_{J_l}}.
\]

These identifications give rise to a family of tensor-product matrix norms: for any partition $\mathcal{J} \in P_d$, define a norm on the space \eqref{idf} by
\[
	\norm{x}_{\mathcal{J}} \coloneqq \norm{x^{(1)} \otimes \ldots \otimes x^{(k)}}_{\mathcal{J}} \coloneqq \max_{i = 1, \ldots, k} \norm{x^{(i)}}_2.
\]
Now, we may define $\norm{A}_{\mathcal{J}}$ as the the operator norm with respect to $\norm{\cdot}_{\mathcal{J}}$:
\begin{equation}\label{tensornorms}
	\norm{A}_{\mathcal{J}} = \sup_{\norm{x}_{\mathcal{J}} \le 1} \abs{Ax}.
\end{equation}
This family of tensor norms agrees with the definitions in \cite{La06} and \cite{AW15} (among others).

Next we extend these definitions to a family of norms $\norm{A}_{\mathcal{J}}$ where $A$ is a $d$-tensor but $\mathcal{J} \in P_{qd}$ for some $q \in \mathbb{N}$. To this end, we first embed $A$ into the space of $qd$-tensors. Indeed, denote by $e_q(A)$ the $qd$-tensor given by
\begin{equation}\label{embed}
(e_q(A))_\mathbf{i} \coloneqq \begin{cases}
a_{i_{1}i_{q+1}i_{2q+1} \ldots i_{(k-1)q+1}} & \text{if} \ i_{kq+j} = i_{kq+1} \ \forall k = 0, \ldots, d-1 \ \forall j = 2, \ldots, q \\
0 & \text{else.}
\end{cases}
\end{equation}
In other words, we divide $\mathbf{i} \in [n]^{qd}$ into $d$ consecutive blocks with $q$ indices in each block $(i_1, \ldots, i_q), (i_{q+1}, \ldots, \linebreak[3] i_{2q}), \ldots$ and only consider such indices for which all elements of these blocks take the same value. In fact, this is an intersection of $d$ ``generalized diagonals''. Now we set
\begin{equation}\label{normgen}
\lVert A \rVert_\mathcal{J} \coloneqq \lVert e_q(A) \rVert_{\mathcal{J}}.
\end{equation}
For $q=1$, this definition trivially agrees with \eqref{tensornorms}.

\begin{bemerkung}\label{monotone}
The norms \eqref{normgen} are monotone with respect to the underlying partition in the following sense. For any two partitions $\mathcal{I} = \{I_1, \ldots, I_\mu \}$ and $\mathcal{J} = \{J_1, \ldots, J_\nu \}$ of $[qd]$, we say that $\mathcal{I}$ is finer than $\mathcal{J}$ (and write $\mathcal{I} \preccurlyeq \mathcal{J}$) if for any $j = 1, \ldots, \mu$ there is a $k \in \{1, \ldots, \nu \}$ such that $I_j \subseteq J_k$. If $\mathcal{I} \preccurlyeq \mathcal{J}$, we have $\lVert A \rVert_\mathcal{I} \le \lVert A \rVert_\mathcal{J}$.
In particular, we always have
\begin{equation}\label{normsmon}
\lVert A \rVert_{\{\{1\}, \ldots, \{qd\} \}} \le \lVert A \rVert_\mathcal{J} \le \lVert A \rVert_{\{1, \ldots, qd \}}.
\end{equation}
\end{bemerkung}

In view of \eqref{normsmon}, the two ``extreme'' norms corresponding to the coarsest and the finest partition of $[qd]$ deserve special attention. Firstly, it is elementary that 
\begin{equation}\label{pHS}
\lVert A \rVert_{\{1,\ldots,qd\}} = \lVert e_q(A) \rVert_{\mathrm{HS}} = \lVert A \rVert_{\mathrm{HS}} = \Big(\sum_{\mathbf{i} \in [n]^d} a_\mathbf{i}^2\Big)^{1/2}.
\end{equation}
Here, $\lVert \cdot \rVert_{\mathrm{HS}}$ denotes the Hilbert--Schmidt norm. Secondly, we have
$$\lVert A \rVert_{\{\{1\}, \ldots, \{qd\} \}} = \lVert e_q(A) \rVert_{\mathrm{op}} = \begin{cases}\norm{A}_\mathrm{op} & q=1\\ \max_{i,j} \abs{a_{ij}} & q \ge 2\end{cases},$$
see Lemma \ref{lemma:1...qdnorm}.

To prove Theorem \ref{thmch}, we furthermore need some auxiliary results. The first one compares the moments of sums of random variables with finite Orlicz norms to moments of Gaussian polynomials and the second one provides the estimates for multilinear forms in Gaussian random variables.

\begin{lemma}[\cite{AW15}, Lemma 5.4]\label{comp}
For any positive integer $k$ and any $p \ge 2$, if $Y_1, \ldots, Y_n$ are independent symmetric random variables with $\lVert Y_i \rVert_{\psi_{2/k}} \le M$, then
$$\big\lVert \sum_{i=1}^n a_iY_i \big\rVert_p \le C_k M \big\lVert \sum_{i=1}^n a_i g_{i1} \cdots g_{ik}\big\rVert_p,$$
where $g_{ij}$ are i.i.d. $\mathcal{N}(0,1)$ variables.
\end{lemma}

\begin{satz}[\cite{La06}, Theorem 1]     \label{Lat}
Let $A = (a_\mathbf{i})_{\mathbf{i} \in [n]^d}$ be a $d$-tensor, and let $G_1, \ldots, G_d$ be i.i.d. standard Gaussian random variables in $\mathbb{R}^n$. Then, for every $p \ge 2$,
$$C_d^{-1} \sum_{\mathcal{J} \in P_{d}} p^{|\mathcal{J}|/2} \lVert A \rVert_\mathcal{J} \le \lVert \langle A, G_1 \otimes \ldots \otimes G_d \rangle \rVert_p \le C_d \sum_{\mathcal{J} \in P_{d}} p^{|\mathcal{J}|/2} \lVert A \rVert_\mathcal{J}.$$
\end{satz}

In the proof of Theorem \ref{thmch}, we actually show $L^p$-estimates for $f_{d,A}(X)$. The following proposition provides the link to concentration inequalities. It was originally proven by Adamczak in \cite{Ad06} and \cite{AW15}, while at this point we cite it in the form given in \cite{SS18}, with a small modification to adjust the constant in front of the exponential.

\begin{proposition}\label{normtoconc}
Assume that a random variable $Z$ satisfies for every $p \ge 2$
$$\lVert Z - \mathbb{E}Z \rVert_p \le \sum_{k=1}^{d} (C_kp)^{k/2}$$
for some constants $C_1, \ldots, C_d \ge 0$. Let $L \coloneqq |\{l \colon C_l > 0 \}|$ and $r \coloneqq \min \{l \in \{1, \ldots, d \} \colon C_l > 0\}$. Then, for any $t > 0$
$$\mathbb{P}(\abs{Z - \mathbb{E}Z} \ge t) \le 2 \exp\left(-\frac{\log(2)}{2(L e)^{2/r}} \min_{k=1, \ldots, d} \left\{\frac{t^{2/k}}{C_k}\right\} \right).$$
\end{proposition}

Now we are able to prove Theorem \ref{thmch}.

\begin{proof}[Proof of Theorem \ref{thmch}]
For simplicity, we always write $f(X) \coloneqq f_{d,A}(X)$. Moreover, without loss of generality, we may assume the $X_i$ to be centered.

Let $X^{(1)}, \ldots, X^{(d)}$ be independent copies of the random vector $X$. Take a set of i.i.d. Rademacher variables $(\varepsilon_i^{(j)})$, $i \le n$, $j \le d$, which are independent of the $(X^{(j)})_j$. By standard decoupling and symmetrization inequalities (see \cite[Theorem 3.1.1]{PG99} and \cite[Lemma 1.2.6]{PG99}),
\begin{align*}
\lVert f(X) \rVert_p
\le C_d \Big\|\sum_{\textbf{i} \in [n]^{\underline{d}}} a_{i_1,\ldots,i_d} X_{i_1}^{(1)} \cdots X_{i_d}^{(d)}\Big\|_p 
\le C_d \Big\|\sum_{\textbf{i} \in [n]^{\underline{d}}} a_{i_1,\ldots,i_d} \varepsilon_{i_1}^{(1)}X_{i_1}^{(1)} \cdots \varepsilon_{i_d}^{(d)}X_{i_d}^{(d)}\Big\|_p.
\end{align*}
An iteration of Lemma \ref{comp} together with $\lVert X_i \rVert_{\psi_{2/q}} \le M$ hence leads to
$$\lVert f(X) \rVert_p
\le \ C_d M^d \Big\|\sum_{\textbf{i} \in [n]^{\underline{d}}} a_{i_1,\ldots,i_d} (g_{i_1,1}^{(1)} \cdots g_{i_1,q}^{(1)}) \cdots (g_{i_d,1}^{(d)} \cdots g_{i_d,q}^{(d)})\Big\|_p.$$
Here, $(g_{i,k}^{(j)})$ is an array of i.i.d. standard Gaussian random variables. Rewriting (recall \eqref{embed}) and applying Theorem \ref{Lat} yields
\begin{align*}
\lVert f(X) \rVert_p
\le \ C_d M^d \lVert \langle e_q(A), \otimes_{j=1}^d \otimes_{k=1}^{q} (g_{i,k}^{(j)})_{i\le n} \rangle \rVert_p
\le \ C_d M^d \sum_{\mathcal{J} \in P_{qd}} p^{|\mathcal{J}|/2} \lVert A \rVert_{\mathcal{J}}.
\end{align*}
The proof is now easily completed by applying Proposition \ref{normtoconc}.
\end{proof}

\section{Hanson--Wright-type inequality: Proof of Proposition \ref{proposition:QuadraticFormWithDiagonal}}\label{section:QuadraticForms}
The main task in the proof of Proposition \ref{proposition:QuadraticFormWithDiagonal} is explicitly calculating the norms.

\begin{lemma}\label{lemma:1...qdnorm}
For any $d$-tensor $A$ and $q \ge 2$ we have
\[
\norm{A}_{\{\{1\}, \ldots, \{qd\}\}} = \norm{A}_\infty = \max_{i_1, \ldots, i_d} \abs{a_{i_1,\ldots, i_d}}.
\]
\end{lemma}

\begin{proof}
Write $\mathcal{J} = \{\{1\},\ldots, \{qd \}\}$. We have
\begin{align*}
\norm{A}_{\mathcal{J}} &= \sup \Big\lbrace \babs{\sum_{i_1 \ldots i_{qd}} (e_q(A))_{i_1, \ldots, i_{qd}} x^1_{i_1} \cdots x^{qd}_{i_{qd}} } : \abs{x^j} \le 1 \text{ for all } j = 1,\ldots,qd \Big\rbrace \\
&= \sup \Big \lbrace \babs{\sum_{i_1,\ldots,i_d} a_{i_1,\ldots, i_d} x^1_{i_1} \cdots x^q_{i_1} x^{q+1}_{i_2} \cdots x^{2q}_{i_1} \cdots x^{(d-1)q+1}_{i_d} \cdots x^{qd}_{i_d}} : \abs{x^j} \le 1 \Big \rbrace \\
&\le \norm{A}_\infty \sup \Big\lbrace \sum_{i_1, \ldots, i_d} \abs{x^1_{i_1} x^2_{i_1}} \cdots \abs{x^{(d-1)q+1}_{i_d} x^{(d-1)q+2}_{i_d}} : \abs{x^j} \le 1 \Big\rbrace \\
&\le \norm{A}_\infty.
\end{align*}
In the third step, we have iteratively used that for $x^{j}$ with $\abs{x^j} \le 1$ we also have $\abs{x_i^j} \le 1$, and applied the Cauchy--Schwarz inequality $d$ times. \par
To obtain the lower bound, let $l_1, \ldots, l_d$ be the index which achieves the maximum. Let $x^1 = \ldots = x^{q} = \delta_{l_1}$, $x^{q+1} = \ldots = x^{2q} = \delta_{l_2}$ and so on, so that
$$\norm{A}_{\mathcal{J}} \ge \abs{a_{l_1 \cdots l_d}} = \norm{A}_\infty.$$
\end{proof}

The following easy observation helps in calculating the norms $\norm{\cdot}_{\mathcal{J}}$. For any partition $\mathcal{J} = \{ J_1, \ldots, J_k \} \in P_{[qd]}$ we write $\tilde{\mathcal{J}} = \{ \tilde{J}_1, \ldots, \tilde{J}_k \}$ for
\begin{equation}\label{eqn:setstildeJ}
\tilde{J}_j = \{ i \in \{1,\ldots, d\} : J_j \cap \{q(i-1)+1,\ldots,qi\} \neq \emptyset \}.
\end{equation}
That is, the sets $\tilde{J}_j$ indicate which of the $d$ $q$-blocks intersect $J_j$. Note that $\cup_j \tilde{J}_j = [d]$, but $\tilde{\mathcal{J}}$ need not be a partition of $[d]$. In fact, some sets $I$ may even appear more than once (with a slight abuse of notation, we choose to keep the set notation in this case anyway). Note that Remark \ref{monotone} extends from partitions to decompositions (all definitions remain valid, even in case of some sets appearing multiple times). Nevertheless, we have by definition
\begin{equation}\label{eqn:ProjectionEquality}
\norm{A}_{\mathcal{J}} = \norm{A}_{\tilde{J}} \coloneqq \sup \bbrace{\sum_{i_1, \ldots, i_d} a_{i_1 \ldots i_d} \prod_{j = 1}^k x_{\mathbf{i}_{\tilde{J}_j}}^{(j)} : \norm{x_{\mathbf{i}_{\tilde{J}_j}}^{(j)}}_2 \le 1},
\end{equation}
i.\,e. the norm does not depend on $\mathcal{J}$, but on its ``projection'' $\tilde{\mathcal{J}}$. We will use this observation in the next lemma to calculate the norms $\norm{A}_\mathcal{J}$ for quadratic forms (i.\,e. $d=2$) and any $q \ge 2$.

\begin{lemma}\label{lemma:restlichenormen}
	Let $A$ be a symmetric matrix, $q \ge 2$ and $\mathcal{J}$ be a partition of $[2q]$. 
	\begin{enumerate}
		\item If $\tilde{\mathcal{J}}$ contains $\{1,2\}$ two or more times, then $\norm{A}_{\mathcal{J}} = \norm{A}_\infty$.
		\item If $\tilde{\mathcal{J}}$ contains $\{1,2\}$ and $\{1\}$ and $\{2\}$, then $\norm{A}_{\mathcal{J}} = \norm{A}_\infty$.
		\item If $\tilde{\mathcal{J}} = \{ \{1,2\}, \{1\}, \ldots, \{1\} \}$ or $\tilde{\mathcal{J}} = \{ \{1,2\}, \{2\}, \ldots, \{2\} \}$, then $\norm{A}_{\mathcal{J}} = \max_i \norm{A_{i\cdot}}_2$.
		\item If $\tilde{\mathcal{J}}$ comprises $l$ times $\{1 \}$ and $k$ times $\{2 \}$ for $k \ge 2, l \ge 2$, then $\norm{A}_{\mathcal{J}} = \norm{A}_\infty$. On the other hand, if $l = 1, k \ge 2$ or $k = 1, l \ge 2$ we have $\norm{A}_{\mathcal{J}} = \max_i \norm{A_{i \cdot}}_2$.
    \item If $\tilde{\mathcal{J}} = \{ \{1\}, \{2\} \}$, then $\norm{A}_{\mathcal{J}} = \norm{A}_{\mathrm{op}}$. 
    \item We have $\norm{A}_{\{[qd]\}} = \norm{A}_{\mathrm{HS}}$.
	\end{enumerate}
	
\end{lemma}

\begin{proof}
	To see (1), write $\tilde{\mathcal{J}} = \{\tilde{J}_1, \ldots, \tilde{J}_k \}$, use the triangle inequality and the fact that $\norm{x}_\infty \le \norm{x}_{\mathrm{HS}}$ for any tensor $x$:
\[
  \norm{A}_{\mathcal{J}} = \sup \bbrace{\sum_{i,j} a_{ij} \prod_{k =1}^l x_{\mathbf{i}_{\tilde{J}_k}}^{(k)}} \le \norm{A}_\infty \sup \bbrace{\sum_{i,j} \abs{x_{ij}} \abs{y_{ij}}} \le \norm{A}_\infty,
\]
where the supremum is taken over all unit vectors $x^{(k)}$. The lower bound follows from \eqref{normsmon} and Lemma \ref{lemma:1...qdnorm}.\par
	(2) follows immediately from $\tilde{\mathcal{J}} \preccurlyeq \{\{1,2\}, \{1,2\} \}$. \par
	(3) follows from the triangle and Cauchy--Schwarz inequality:
	\begin{align*}
	\norm{A}_{\mathcal{J}} &\le \sup \bbrace{\sum_i \abs{\prod_{k = 1}^l y^k_i} \abs{\sum_j a_{ij} x_{ij}}} \le \sup \bbrace{\sum_i \abs{\prod_{k = 1}^l y_i^k} \norm{(a_{ij})_j}_2 \norm{x_{i\cdot}}_2} \\
	&\le \max_i \norm{(a_{ij})_j}_2 \sup \bbrace{\abs{\prod_{k = 1}^l y_i^k} \norm{x_{i \cdot}}_2} \le \max_i \norm{(a_{ij})_j}_2.
	\end{align*}
	The lower bound is obtained by choosing $y^1, \ldots, y^l$ as a Dirac delta on the row for which $\max_i \norm{A_{i \cdot}}$ is attained. \par
	To see (4), note that the case $k \ge 2, l \ge 2$ is very similar to the second part. If $l = 1, k \ge 2$ or $k = 1, l \ge 2$, similar arguments as in the third part give for any $x, y^1, \ldots, y^l$ with norm at most one
	\[
	\abs{\sum_{i,j} a_{ij} x_i \prod_k y^k_j} \le \sum_j \abs{\prod_k y_j^k} \abs{\sum_i a_{ij} x_i} \le \sum_j \abs{\prod_k y^k_j \norm{(a_{ij})_j}_2} \le \max_i \norm{(a_{ij})_j}_2.
	\]
	The lower bound again follows by choosing suitable Dirac deltas. \par
  $(5)$ and $(6)$ are obvious from the definitions.
\end{proof}

Actually, we have the equality
\[
\max_{i = 1,\ldots,n} \norm{(a_{ij})_j}_2 = \norm{A}_{2 \to \infty},
\]
where $\left\lVert A \right\rVert_{p \to q} \coloneqq \sup \left \lbrace \norm{Ax}_q : \norm{x}_p \le 1 \right\rbrace.$ For the proof, see \cite[Proposition 6.1]{CTP17}. Especially this yields $\max_{i = 1,\ldots,n} \norm{(a_{ij})_j}_2 \le \norm{A}_{\mathrm{op}}$.

We are now ready to prove Proposition \ref{proposition:QuadraticFormWithDiagonal}. Throughout the rest of this section, for a matrix $A$ let us denote by $A^{\mathrm{od}}$ its off-diagonal and by $A^{\mathrm{d}}$ the diagonal part.

\begin{proof}[Proof of Proposition \ref{proposition:QuadraticFormWithDiagonal}]
Lemma \ref{lemma:restlichenormen} shows that we only need to consider the four norms $\norm{A}_{\mathrm{HS}}, \norm{A}_{\mathrm{op}}, \max_{i  = 1,\ldots,n} \norm{(a_{ij})_j}_2$ and $\norm{A}_\infty$. 
It is easy to see that $\norm{A}_{\mathrm{HS}} \ge \norm{A}_{\mathrm{op}} \ge \max_i \norm{(a_{ij})_j}_2 \ge \norm{A}_\infty$. Thus, we need to determine which partitions give rise to which norms. 

The only partition producing the Hilbert--Schmidt norm is $\mathcal{J}_1 = \{ [qd] \}$, with $\abs{\mathcal{J}_1} = 1$. The operator norm appears for the decomposition $\mathcal{J}_2 = \{ \{1,\ldots,q\}, \{q+1, \ldots, 2q\} \}$ with $\abs{\mathcal{J}_2} = 2$. Moreover, it is easy to see that all partitions $\mathcal{J}_3$ of $[2q]$ giving rise to $\max_{i = 1,\ldots,n} \norm{(a_{ij})_j}_2$ satisfy $\abs{\mathcal{J}_3} \in \{2,\ldots,q+1\}$. Finally, for all $k = 2,\ldots, 2q$ there are partitions $\mathcal{J}_4$ such that $\norm{A}_{\mathcal{J}_4} = \norm{A}_\infty$.

Hence for a diagonal-free matrix $A$ we have by simply plugging in the norms calculated in Lemmas \ref{lemma:1...qdnorm} and \ref{lemma:restlichenormen} into Theorem \ref{thmch}
\begin{align}\label{eqn:inequalityForDiagonalFree}
\mathbb{P}\Big( \Big \lvert \sum_{i,j} a_{ij}(X_i X_j - \IE X_i \IE X_j) \Big\rvert \ge t \Big) \le 2 \exp \Big( - \frac{1}{C} \eta(A,q,t/M^2) \Big),
\end{align}
where
\pagebreak[3]
\begin{align*}
\eta(A,q,t) &= \min \left( \frac{t^2}{\norm{A}_{\mathrm{HS}}^2}, \frac{t}{\norm{A}_{\mathrm{op}}}, \min_{l = 2,\ldots, q+1} \Big( \frac{t}{\max_i \norm{(a_{ij})_j}_2} \Big)^{\frac{2}{l}} , \min_{l = 2,\ldots, 2q} \Big( \frac{t}{\norm{A}_{\infty}} \Big)^{\frac{2}{l}} \right) \notag \\
&=  \min \left( \frac{t^2}{\norm{A}_{\mathrm{HS}}^2}, \frac{t}{\norm{A}_{\mathrm{op}}}, \Big( \frac{t}{\max_i \norm{(a_{ij})_j}_2} \Big)^{\frac{2}{q+1}}, \Big( \frac{t}{\norm{A}_{\infty}} \Big)^{\frac{1}{q}} \right).
\end{align*}
In the last two terms, we can choose the largest $l$ since we can assume that $\frac{t}{\norm{A}_{\mathcal{J}}} \ge 1$ for any partition $\mathcal{J}$, as the minimum is achieved in $\frac{t^2}{\norm{A}_{\mathrm{HS}}^2}$ otherwise.

For matrices with non-vanishing diagonal, we divide the quadratic form into an off-diagonal and a purely diagonal part, i.\,e.
\[
\sum_{i,j} a_{ij} X_i X_j = \sum_{i,j} a^{\mathrm{od}}_{ij} X_i X_j + \sum_{i = 1}^n a^{\mathrm{d}}_{ii} X_i^2.
\]
For brevity, let us define $P(t) \coloneqq \mathbb{P}\Big(\big\lvert \sum_{i,j} a_{ij} X_i X_j - \sum_{i = 1}^n \sigma_i^2 a_{ii} \big \rvert \ge t\Big).$ Use the above decomposition and the subadditivity to obtain
\[
P(t) \le \mathbb{P}\Big( \abs{\sum_{i,j} a^{\mathrm{od}}_{ij} X_i X_j} \ge t/2 \Big) + \mathbb{P}\Big( \abs{\sum_{i =1 }^n a^{\mathrm{d}}_{ii} (X_i^2 - \sigma_i^2)} \ge t/2 \Big) \eqqcolon p_1(t) + p_2(t).
\]
Equation \eqref{eqn:inequalityForDiagonalFree} can be used to upper bound $p_1(t)$ as
\begin{align}			\label{eqn:UpperBoundOffDiag}
p_1(t) \le 2 \exp \left( - \frac{1}{C_2} \eta(A^{od},q,t/M^2) \right).
\end{align}
The diagonal term can be treated by applying Theorem \ref{thmch} for $d = 1$, $q = 4$ and $a = (A_{ii})_{i = 1,\ldots,n}$. Moreover, it is easy to see that we have $\norm{a}_{\{1,2,3,4\}} = \sum_i (a^{\mathrm{d}}_{ii})^2$ (cf. \eqref{pHS}) and $\norm{a}_{\mathcal{J}} = \norm{A^{\mathrm{d}}}_\infty$ for any other decomposition $\mathcal{J}$. Consequently,
\begin{align}			\label{eqn:UpperBoundDiag}
p_2(t) &\le 2 \exp \Big( - \frac{1}{C_1} \min\Big( \frac{t^2}{\norm{A^{\mathrm{d}}}^2_{\mathrm{HS}}}, \frac{t}{\norm{A^{\mathrm{d}}}_\infty}, \Big(\frac{t}{\norm{A^{\mathrm{d}}}_\infty} \Big)^{2/3}, \Big( \frac{t}{\norm{A^{\mathrm{d}}}_\infty} \Big)^{1/2} \Big) \Big) \\
&= 2 \exp \Big( - \frac{1}{C_1} \eta_{1,A^{\mathrm{d}}}(t) \Big).
\end{align}
Thus, by combining \eqref{eqn:UpperBoundOffDiag} and \eqref{eqn:UpperBoundDiag} we have
\begin{align*}
	P(t) 
  \le 4 \exp \left( -C \min(\eta(A^{\mathrm{od}},q,t), \eta_{1,A^{\mathrm{d}}}(t) \right).
\end{align*}
Now it remains to lower bound the minimum by grouping the terms according to the different powers of $t$. This gives
\[
p(t) \le 4 \exp\Big( - \frac{1}{C} \tilde{\eta}(A,q,t/M) \Big),
\]
where
{\scriptsize
\[
\tilde{\eta}(A,q,t) \coloneqq \min \left( \frac{t^2}{\norm{A}^2_{\mathrm{HS}}}, \frac{t}{\max(\norm{A^{\mathrm{od}}}_{\mathrm{op}}, \norm{A^\mathrm{d}}_\infty)}, \Big( \frac{t}{\max_{i = 1,\ldots,n} \norm{(a_{ij})_j}_2} \Big)^{2/(q+1)}, \left( \frac{t}{\norm{A}_\infty} \right)^{1/q} \right).
\]}Lastly, from the characterization $\norm{A}_{\mathrm{op}} \coloneqq \sup_{x \in S^{n-1}} \abs{\skal{x,Ax}}$ it can be easily seen that the inequalities $\norm{A^{\mathrm{d}}}_{\infty} \le \norm{A}_{\mathrm{op}} $ and $\norm{A^{\mathrm{od}}}_{\mathrm{op}} \le 2 \norm{A}_{\mathrm{op}}$ hold, and the constant $4$ can be changed to $2$ by adjusting the constant in the exponent.
\end{proof}

\section{The polynomial case: Proof of Theorem \ref{thmpol}}\label{section:ProofThm16}
Let us now treat the case of general polynomials $f(X)$ of total degree $D \in \mathbb{N}$. Before we start, we need to discuss some more properties of the norms $\lVert A \rVert_{\mathcal{J}}$. To this end, recall the Hadamard product of two $d$-tensors $A, B$ given by $A \circ B \coloneqq (a_\mathbf{i} b_\mathbf{i})_{\mathbf{i} \in [n]^d}$ (pointwise multiplication). If we interpret a $d$-tensor as a function $[n]^d \to \mathbb{R}$, we may define ``indicator matrices'' $1_C$ for a set $C \subseteq [n]^d$ by setting $1_C = (a_{\textbf{i}})_{\textbf{i}}$ with $a_{\textbf{i}} = 1$ if $\textbf{i} \in C$ and $a_{\textbf{i}} = 0$ otherwise. If $|\mathcal{J}| > 1$, we do not have
\begin{equation}\label{uncond}
\lVert A \circ 1_C \rVert_{\mathcal{J}} \le \lVert A \rVert_{\mathcal{J}}
\end{equation}
in general. However, \cite[Lemma 5.2]{AW15} shows a number of situations in which such an inequality does hold.

\begin{lemma}\label{uncondlemma}
	Let $A = (a_{\textbf{i}})_{\mathbf{i} \in [n]^d}$ be a $d$-tensor.
	\begin{enumerate}
		\item If $C = \{\mathbf{i} \colon i_{k_1} = j_1, \ldots, i_{k_l} = j_l \}$ for some $1 \le k_1 < \ldots < k_l \le d$ (``generalized row''), then \eqref{uncond} holds.
		\item If $C = \{\mathbf{i} \colon i_k = i_l \ \forall k, l \in K \}$ for some $K \subset [d]$ (``generalized diagonal''), then \eqref{uncond} holds.
		\item If $C_1, C_2 \subset [n]^d$ are such that \eqref{uncond} holds, then so is $C_1 \cap C_2$.
	\end{enumerate}
\end{lemma}

There is a further situation in which a version of \eqref{uncond} holds. Indeed, for any partition $\mathcal{K} = \{K_1, \ldots, K_a \} \in P_d$ of $[d]$ we define
\begin{equation}\label{LK}
L(\mathcal{K}) = \{\textbf{i} \in [n]^{d} \colon i_k = i_l \Leftrightarrow \exists j \colon k, l \in K_j \}.
\end{equation}
That is, $L(\mathcal{K})$ is the set of those indices for which the partition into level sets is equal to $\mathcal{K}$.

\begin{lemma}\label{lemLK}
	Let $\mathcal{J} \in P_{qd}$, $\mathcal{K} \in P_d$ and $A$ be a $d$-tensor. Then,
	$$\lVert A \circ 1_{L(\mathcal{K})} \rVert_{\mathcal{J}} \le 2^{|\mathcal{K}|(|\mathcal{K}|-1)/2}\lVert A \rVert_{\mathcal{J}}.$$
\end{lemma}

\begin{proof}
This is a generalization of \cite[Corollary 5.3]{AW15} which corresponds to the case $q = 1$. First note that by definition,
	$$\lVert A \circ 1_{L(\mathcal{K})} \rVert_{\mathcal{J}} = \lVert e_q(A \circ 1_{L(\mathcal{K})}) \rVert_{\mathcal{J}} = \lVert e_q(A) \circ e_q(1_{L(\mathcal{K})}) \rVert_{\mathcal{J}}.$$
	Therefore, it suffices to prove that for any $qd$-tensor $B$,
	$$\lVert B \circ e_q(1_{L(\mathcal{K})}) \rVert_{\mathcal{J}} \le 2^{|\mathcal{K}|(|\mathcal{K}|-1)/2}\lVert B \rVert_{\mathcal{J}}.$$
	To see this, observe that $e_q(1_{L(\mathcal{K})})$ is the indicator matrix of a set $C$ which can be written as an intersection of $|\mathcal{K}|$ generalized diagonals (with the cardinality of the underlying sets of indices in \eqref{gendiag} always being an integer multiple of $q$) and $|\mathcal{K}|(|\mathcal{K}|-1)/2$ sets of the form $\{\mathbf{i}: i_{kq +1} \ne i_{lq +1} \}$ for $k < l$. Recall that
	$$\lVert B \circ 1_{\{i_{kq+1} \ne i_{lq+1} \}} \rVert_\mathcal{J} = \lVert B - B \circ 1_{\{i_{kq+1} = i_{lq+1} \}} \rVert_\mathcal{J} \le 2 \lVert B \rVert_\mathcal{J},$$
	using Lemma \ref{uncondlemma} (2) in the last step. As a consequence, the claim follows by applying Lemma \ref{uncondlemma} (2) again and a generalization of Lemma \ref{uncondlemma} (3).
\end{proof}

Finally, it remains to note that \cite[Lemma 5.1]{AW15} can be generalized as follows.

\begin{lemma}\label{Absch}
Let $A$ be a $d$-tensor, and let $v_1, \ldots, v_d \in \mathbb{R}^n$ be any vectors. Then, for any partition $\mathcal{J} \in P_{qd}$,
$\lVert A \circ \otimes_{i=1}^d v_i \rVert_\mathcal{J} \le \lVert A \rVert_\mathcal{J} \prod_{i=1}^d \lVert v_i \rVert_\infty.$
\end{lemma}

\begin{proof}
Recall equations \eqref{eqn:setstildeJ} and \eqref{eqn:ProjectionEquality}. We have
\begin{align*}
\norm{A \circ \otimes_{i = 1}^d v_i}_{\mathcal{J}} &= \sup \bbrace{\sum_{i_1, \ldots, i_{qd}} (e_q(A))_{i_1 \ldots i_{qd}} (e_q(\otimes_{i = 1}^d v_i))_{i_1 \ldots i_{qd}} \prod_{j = 1}^k x_{\mathbf{i}_{J_j}}^{(j)} : \norm{x_{\mathbf{i}_{J_j}}^{(j)}}_2 \le 1} \\
&= \sup \bbrace{\sum_{i_1, \ldots, i_d} a_{i_1 \ldots i_d} v_1^{i_1} \cdots v_d^{i_d} \prod_{j = 1}^k x_{\mathbf{i}_{\tilde{J}_j}}^{(j)} : \norm{x_{\mathbf{i}_{\tilde{J}_j}}^{(j)}}_2 \le 1} \\
&\le \sup \bbrace{\sum_{i_1, \ldots, i_d} a_{i_1 \ldots i_d} \prod_{j = 1}^k x_{\mathbf{i}_{\tilde{J}_i}}^{(j)} : \norm{x_{\mathbf{i}_{\tilde{J}_j}}^{(j)}}_2 \le 1} \prod_{i = 1}^d \norm{v_i}_\infty\\
&= \sup \bbrace{\sum_{i_1, \ldots, i_{qd}} (e_q(A))_{i_1 \ldots i_{qd}} \prod_{j = 1}^k x_{\mathbf{i}_{J_j}}^{(j)} : \norm{x_{\mathbf{i}_{J_j}}^{(j)}}_2 \le 1} \prod_{i = 1}^d \norm{v_i}_\infty  \\
&= \norm{A}_{\mathcal{J}} \prod_{i = 1}^d \norm{v_i}_\infty.
\end{align*}
To see the third step, for each $v_l$ we choose a set $\mathcal{J}_j$ such that $l \in \mathcal{J}_j$ and then define vectors $\tilde{x}_{\mathbf{i}_{\tilde{J}_j}}^{(j)}$ by multiplying $x_{\mathbf{i}_{\tilde{J}_j}}^{(j)}$ by the components of the vectors $v_l$ which were attributed to $\mathcal{J}_j$. In particular, this leads to $\norm{\tilde{x}_{\mathbf{i}_{\tilde{J}_j}}^{(j)}}_2 \le \prod_{l} \norm{v_l}_\infty \norm{x_{\mathbf{i}_{\tilde{J}_j}}^{(j)}}_2$, where the product is taken over all the vectors $v_l$ which were attributed to $x_{\mathbf{i}_{\tilde{J}_j}}^{(j)}$.
\end{proof}

Before we begin with the proof of the concentration results for general polynomials, let us give some definitions. Boldfaced letters will always represent a vector (mostly a multiindex with integer components), and for any vector $\mathbf{i}$ let $\abs{\mathbf{i}} \coloneqq \sum_j i_j$. For the sake of brevity we define
\begin{align*}
I_{m,d} &\coloneqq \{ (i_1, \ldots, i_m) \in \IN^m : \abs{\mathbf{i}} = d \}, \\
I_{m,\le d} &\coloneqq \{ (i_1, \ldots, i_m) \in \IN^m : \abs{\mathbf{i}} \le d \}.
\end{align*}
Given two vectors $\mathbf{i}, \mathbf{k}$ of equal size, we write $\mathbf{k} \le \mathbf{l}$ if $k_j \le l_j$ for all $j$, and $\mathbf{k} < \mathbf{l}$ if $\mathbf{k} \le \mathbf{l}$ and there is at least one index such that $k_j < l_j$. Lastly, by $f \lesssim g$ we mean an inequality of the form $f \le C_{D,q} g$.

\begin{proof}[Proof of Theorem \ref{thmpol}]
We assume $M = 1$. For the general case, given random variables $X_1,\ldots, X_n$ with $\norm{X_i}_{\Psi_{2/q}} \le M$, define $Y_i \coloneqq M^{-1} X_i$. The polynomial $f = f(X)$ can be written as a polynomial $\tilde{f} = \tilde{f}(Y)$ by appropriately modifying the coefficients, i.\,e. multiplying each monomial by $M^{r}$, where $r$ is its total degree. Now it remains to see that $\partial_{i_1 \ldots i_j} \tilde{f}(Y) = M^j \partial_{i_1 \ldots i_j} f(X)$.

\textbf{Step 1.} First, we reduce the problem to generalizations of chaos-type functionals \eqref{eqn:chaos}. Indeed, by sorting according to the total grade, $f$ may be represented as
$$f(x) = \sum_{d=1}^D \sum_{\nu=1}^d \sum_{\textbf{k} \in I_{\nu,d}} \sum_{\textbf{i} \in [n]^{\underline{\nu}}} c_{(i_1,k_1), \ldots, (i_\nu,k_\nu)}^{(d)} x_{i_1}^{k_1} x_{i_2}^{k_2} \cdots x_{i_\nu}^{k_\nu} + c_0,$$
where the constants satisfy $c_{(i_1,k_1), \ldots, (i_\nu,k_\nu)}^{(d)} = c_{(i_{\pi_1},k_{\pi_1}), \ldots, (i_{\pi_\nu},k_{\pi_\nu})}^{(d)}$ for any permutation $\pi \in \mathcal{S}_\nu$. As in \cite{AW15}, by rearranging and making use of the independence of $X_1, \ldots, X_n$, this leads to the estimate
$$|f(X) - \mathbb{E}f(X)| \le \sum_{d=1}^D \sum_{\nu=1}^d \sum_{\textbf{k} \in I_{\nu,d}} \Big|\sum_{\textbf{i} \in [n]^{\underline{\nu}}} a_{\textbf{i}}^{\textbf{k}} (X_{i_1}^{k_1}- \mathbb{E}X_{i_1}^{k_1}) \cdots (X_{i_\nu}^{k_\nu}-\mathbb{E}X_{i_\nu}^{k_\nu})\Big|,$$
where
\[
a_{\textbf{i}}^{\textbf{k}} = \sum_{m=\nu}^{D} \sum_{\substack{k_{\nu+1},\ldots,k_m>0 \\ k_1+\ldots+k_m \le D}} \sum_{\substack{i_{\nu+1}, \ldots, i_m \\ (i_1, \ldots, i_m) \in [n]^{\underline{m}}}} \binom{m}{\nu} c_{(i_1,k_1), \ldots, (i_m,k_m)}^{(k_1+\ldots+k_m)} \prod_{\alpha=1}^m \mathbb{E}X_{i_\alpha}^{k_{i_\alpha}}.
\]

\textbf{Step 2.} Note that $\lVert X_i^k \rVert_{\psi_{2/(qk)}} = \lVert X_i \rVert_{\psi_{2/q}}^k \le 1$. Thus, slightly modifying the proof of Theorem \ref{thmch} (in particular, also using Lemma \ref{comp} for the non-linear terms), we obtain the estimate
\begin{align*}
&\lVert f(X) - \mathbb{E}f(X) \rVert_p \lesssim \sum_{d=1}^D \sum_{\nu=1}^d \sum_{\textbf{k} \in I_{\nu,d}} \Big\|\sum_{\textbf{i} \in [n]^{\underline{\nu}}} a_{\textbf{i}}^{\textbf{k}} (g_{i_1,1}^{(1)} \cdots g_{i_1,q k_1}^{(1)}) \cdots (g_{i_\nu,1}^{(\nu)} \cdots g_{i_\nu,q k_\nu}^{(\nu)})\Big\|_p.
\end{align*}
Here, $(g_{i,k}^{(j)})$ is an array of i.i.d. standard Gaussian random variables.

Moreover, the family $(a_\mathbf{i}^{\mathbf{k}})_{\nu \in \{1,\ldots, d\}, k \in I_{\nu,d}, i \in [n]^{\underline{\nu}}}$ gives rise to a $d$-tensor $A_d$ as follows. Given any index $\mathbf{i} = (i_1, \ldots, i_d)$ there is a unique number $r \in \{1, \ldots, d \}$ of distinct elements $j_1, \ldots, j_r$ with each $j_l$ appearing exactly $k_l$ times in $\mathbf{i}$. Consequently, we set $a_{i_1 \ldots i_d} \coloneqq a_{j_1,\ldots,j_r}^{(l_1,\ldots,l_r)}$, and $A_d = (a_\mathbf{i})_{\mathbf{i} \in [n]^d}$. Note that this is well-defined due to the symmetry assumption.

For any $\mathbf{k} \in I_{\nu,d}$ denote by $\mathcal{K}(\mathbf{k}) = \mathcal{K}(k_1, \ldots, k_\nu) \in P_d$ the partition which is defined by splitting the set $\{1, \ldots, d \}$ into consecutive intervals of length $k_1, \ldots, k_\nu$. In other words, $\mathcal{K}(\mathbf{k}) = \{K_1, \ldots, K_\nu \}$ with $K_l = \{\sum_{i=1}^{l-1} k_i + 1, \sum_{i=1}^{l-1} k_i + 2, \ldots, \sum_{i=1}^{l} k_i \}$, $l = 1, \ldots, \nu$. Now, recalling the definitions of $e_q$ \eqref{embed} and of $L(\mathcal{K})$ \eqref{LK}, by rewriting and applying Lemma \ref{uncondlemma} we obtain
\begin{equation}\label{eqn:LpNormEstimate}
\begin{split}
\lVert f(X) - \mathbb{E}f(X) \rVert_p &\lesssim \sum_{d=1}^D \sum_{\nu=1}^d \sum_{\textbf{k} \in I_{\nu,d}}  \lVert \langle e_q(A_d \circ 1_{L(\mathcal{K}(\mathbf{k}))}), \otimes_{j=1}^\nu \otimes_{k=1}^{q k_j} (g_{i,k}^{(j)})_{i\le n} \rangle \rVert_p\\
&\lesssim \sum_{d=1}^D \sum_{\nu=1}^d \sum_{\textbf{k} \in I_{\nu,d}} \sum_{\mathcal{J} \in P_{qd}} p^{|\mathcal{J}|/2} \lVert A_d \circ 1_{L(\mathcal{K}(k_1, \ldots, k_\nu))} \rVert_{\mathcal{J}}\\
&\lesssim \sum_{d=1}^D \sum_{\mathcal{J} \in P_{qd}} p^{|\mathcal{J}|/2} \lVert A_d \rVert_{\mathcal{J}}.
\end{split}
\end{equation}

\textbf{Step 3.} Next, we replace $\lVert A_d \rVert_\mathcal{J}$ by $\lVert \mathbb{E} f^{(d)}(X) \rVert_\mathcal{J}$. To this end, first note that for $\mathbf{i} \in [n]^d$ with distinct indices $j_1, \ldots, j_\nu$ which are taken $l_1, \ldots, l_\nu$ times, we have
\begin{align*}
&\mathbb{E} \frac{\partial^d f}{\partial x_{i_1} \ldots \partial x_{i_d}}(X) = \sum_{\mathbf{k} : \mathbf{k} \ge \mathbf{l}} \sum_{m=\nu}^D \sum_{\substack{k_{\nu+1}, \ldots, k_m > 0 \\ k_1 + \ldots + k_m \le D}} \sum_{\substack{j_{\nu+1}, \ldots, j_m \\ (j_1, \ldots, j_m) \in [n]^{\underline{m}}}} \\ &\left(\binom{m}{\nu} \nu! c_{(j_1,k_1), \ldots, (j_m,k_m)}^{(k_1+\ldots+k_m)} \prod_{\alpha=1}^\nu \mathbb{E} X_{j_\alpha}^{k_\alpha-l_\alpha} \prod_{\alpha=\nu+1}^m \mathbb{E}X_{j_\alpha}^{k_\alpha} \prod_{\alpha=1}^\nu \frac{k_\alpha!}{(k_\alpha-l_\alpha)!} \right)\\
&= \nu!l_1! \cdots l_\nu!a_{i_1, \ldots, i_d}  + R_\mathbf{i}^{(d)},
\end{align*}
where the ``remainder term'' $R_\mathbf{i}^{(d)}$ corresponds to the set of indices $\mathbf{k}$ satisfying $\mathbf{k} > \mathbf{l}$. If $d = D$, we clearly have $R_\mathbf{i}^{(d)} = 0$, and therefore
\begin{equation}\label{eqn:fDandA}
\mathbb{E} \frac{\partial^D f}{\partial x_{i_1} \ldots \partial x_{i_D}}(X) = \nu!l_1! \cdots l_\nu! a_{i_1 \ldots i_D} = \nu! |I_1|! \cdots |I_\nu|! a_{i_1 \ldots i_D},
\end{equation}
where $\mathcal{I} = \{I_1, \ldots, I_\nu \}$ is the partition given by the level sets of the index $\mathbf{i}$. It follows that for any partition $\mathcal{J} \in P_{qD}$,
$$\lVert A_D \rVert_\mathcal{J} \le \sum_{\mathcal{K} \in P_D} \lVert A_D \circ 1_{L(\mathcal{K})} \rVert_{\mathcal{J}} \le \sum_{\mathcal{K} \in P_D} \lVert \mathbb{E} f^{(D)}(X) \circ 1_{L(\mathcal{K})} \rVert_{\mathcal{J}} \lesssim \lVert \mathbb{E} f^{(D)}(X) \rVert_{\mathcal{J}},$$
using the partition of unity $1 = \sum_{\mathcal{K} \in P_D} 1_{L(\mathcal{K})}$ and the triangle inequality in the first, equation \eqref{eqn:fDandA} in the second and Lemma \ref{lemLK} in the last step.

The proof is now completed by induction. More precisely, in the next step will show that for any $d \in \{1,\ldots, D-1\}$ and any partitions $\mathcal{I} = \{I_1, \ldots, I_\mu \} \in P_d$, $\mathcal{J} = \{J_1, \ldots, J_\nu \} \in P_{qd}$,
\begin{equation}\label{tbp}
\lVert R^{(d)} \circ 1_{L(\mathcal{I})} \rVert_\mathcal{J} \lesssim \sum_{k=d+1}^D \sum_{\substack{\mathcal{K}\in P_{qk} \\ |\mathcal{K}|=|\mathcal{J}|}} \lVert A_k \rVert_\mathcal{K}.
\end{equation}
Having \eqref{tbp} at hand, it follows by reverse induction and Lemma \ref{lemLK} that
$$\sum_{d=1}^D \sum_{\mathcal{J} \in P_{qd}} p^{|\mathcal{J}|/2} \lVert A_d \rVert_\mathcal{J} \lesssim \sum_{d=1}^D \sum_{\mathcal{J} \in P_{qd}} p^{|\mathcal{J}|/2} \lVert \mathbb{E} f^{(d)}(X) \rVert_\mathcal{J}.$$
Plugging this into \eqref{eqn:LpNormEstimate} and applying Proposition \ref{normtoconc} finishes the proof.

\textbf{Step 4:} To show \eqref{tbp}, let us analyze the ``remainder tensors'' $R^{(d)}$ in more detail. To this end, fix $d \in \{1,\ldots,D-1\}$ and partitions $\mathcal{I} = \{I_1, \ldots, I_\nu \} \in P_d$, $\mathcal{J} = \{J_1, \ldots, J_\mu\} \in P_{qd}$, and let $\mathbf{l}$ be the vector with $l_\alpha \coloneqq |I_\alpha|$ (note that this implies $\abs{\mathbf{l}} = d$). For any $\mathbf{k} \in I_{\nu,\le D}$ with $\mathbf{k} > \mathbf{l}$, we define a $d$-tensor $S_\mathcal{I}^{(d,\mathbf{k})} = (s_\mathbf{i}^{(d,k_1, \ldots, k_\nu)})_{\mathbf{i} \in [n]^d} = (s_\mathbf{i}^{(d)})_{\mathbf{i} \in [n]^d}$ as follows: 
{\scriptsize
\[
s^{(d)}_\mathbf{i} = 
1_{\mathbf{i} \in L(\mathcal{I})}\sum_{m=\nu}^D \sum_{\substack{k_{\nu+1}, \ldots, k_m > 0 \\ k_1 + \ldots + k_m \le D}} \sum_{\substack{j_{\nu+1}, \ldots, j_m \\ (j_1, \ldots, j_m) \in [n]^{\underline{m}}}} \binom{m}{\nu} c_{(j_1,k_1), \ldots, (j_m,k_m)}^{(k_1+\ldots+k_m)} \prod_{\alpha=1}^\nu \mathbb{E} X_{j_\alpha}^{k_\alpha-l_\alpha} \prod_{\alpha=\nu+1}^m \mathbb{E}X_{j_\alpha}^{k_\alpha} 
\]
}
Here, we denote by $j_\alpha$ the value of $\mathbf{i}$ on the level set $I_\alpha$. Clearly,
$$R^{(d)} \circ 1_{L(\mathcal{I})} = \sum_{\substack{\mathbf{k} \in I_{\nu,\le D} \\ \mathbf{k} > \mathbf{l}}} \nu!\frac{k_1}{(k_1 - l_1)!} \cdots \frac{k_\nu}{(k_\nu - l_\nu)!} S_\mathcal{I}^{(d,\mathbf{k})}.$$
Therefore, it remains to prove that there is a partition $\mathcal{K} \in P_{q\abs{\mathbf{k}}}$ with $|\mathcal{K}| = |\mathcal{J}|$ such that
\begin{equation}\label{remtp}
\lVert S_\mathcal{I}^{(d,\mathbf{k})} \rVert_\mathcal{J} \lesssim \lVert A_{\abs{\mathbf{k}}} \rVert_{\mathcal{K}}.
\end{equation}

The tensor will be given by an appropriate embedding of the $d$-tensor $S^{(d,\mathbf{k})}_{\mathcal{I}}$. To this end, choose any partition $\tilde{\mathcal{I}} = \{\tilde{I}_1, \ldots, \tilde{I}_\nu \} \in P_{\abs{\mathbf{k}}}$ with $|\tilde{I}_\alpha| = k_\alpha$ and $I_\alpha \subset \tilde{I}_\alpha$ for all $\alpha$. Embedding the $d$-tensor $S^{(d,\mathbf{k})}_{\mathcal{I}}$ into the space of $\abs{\mathbf{k}}$-tensors is done by defining a new tensor $\tilde{S}^{\abs{\mathbf{k}}} = (\tilde{s}^{\abs{\mathbf{k}}}_\mathbf{i})_\mathbf{i}$ given by
\begin{equation}\label{newt}
\tilde{s}^{\abs{\mathbf{k}}}_\mathbf{i} = s_{\mathbf{i}_{[d]}}^{(d)}1_{\mathbf{i} \in L(\tilde{\mathcal{I}})}.
\end{equation}

We choose the partition $\mathcal{K} = \{K_1, \ldots, K_\mu \}$ defined in the following way: for any $j$, we have $J_j \subset K_j$, so that it remains to assign the elements $r \in \{qd+1, \ldots, q\abs{\mathbf{k}}\}$ to the sets $K_j$. Write $r = \eta q + m$ for some $\eta \in \{d, \ldots, \abs{\mathbf{k}} - 1\}$ and $m \in \{1, \ldots, q \}$. Since $\tilde{\mathcal{I}}$ is a partition of $\abs{\mathbf{k}}$, there is a unique $j \in \{1,\ldots, \nu\}$ such that $\eta + 1 \in \tilde{I}_j$. Take the smallest element $t$ in $\tilde{I}_j$ -- since $I_j \subset \tilde{I}_j$, we have $t \in [d]$ -- and add $r$ to the same set as $\pi(r) \coloneqq (t-1)q + m$.

We claim that
\begin{equation}\label{norm2}
\lVert S_\mathcal{I}^{(d,\abs{\mathbf{k}})} \rVert_\mathcal{J} \le \lVert \tilde{S}^{\abs{\mathbf{k}}} \rVert_\mathcal{K}.
\end{equation}
To see this, let $x^{(\beta)} = (x^{(\beta)}_{\mathbf{i}_{J_\beta}})$, $\beta = 1, \ldots, \mu$, be a collection of vectors satisfying $\lVert x^{(\beta)} \rVert_2 \le 1$. This gives rise to a further collection of unit vectors $y^{(\beta)} = (y^{(\beta)}_{\mathbf{i}_{K_\beta}})$, $\beta = 1, \ldots, \mu$, defined by
$$y^{(\beta)}_{\mathbf{i}_{K_\beta}} = x^{(\beta)}_{\mathbf{i}_{K_\beta \cap [qd]}} \prod_{r\in K_\beta\setminus[qd]} 1_{i_r = i_{\pi(r)}}$$
(recall the definition of $\pi(r)$ given in the paragraph above). Now, it follows that
\begin{align*}
\sum_{|\mathbf{i}_{[d]}| \le n} s_{\mathbf{i}_{[d]}}^{(d)}\prod_{\beta=1}^\mu x_{(e_q(\mathbf{i}))_{J_\beta}}^{(\beta)} = \sum_{|\mathbf{i}_{[\abs{\mathbf{k}}]}| \le n}\tilde{s}_{\mathbf{i}_{\abs{\mathbf{k}}}}^{\abs{\mathbf{k}}} \prod_{\beta=1}^\mu x_{(e_q(\mathbf{i}))_{J_\beta}}^{(\beta)} = \sum_{|\mathbf{i}_{[\abs{\mathbf{k}}]}| \le n}\tilde{s}_{\mathbf{i}_{[\abs{\mathbf{k}}]}}^{(\abs{\mathbf{k}})} \prod_{\beta=1}^\mu y_{(e_q(\mathbf{i}))_{K_\beta}}^{(\beta)}.
\end{align*}
These equations follow from the definition of the matrix $\tilde{S}^{\abs{\mathbf{k}}}$ and the fact that if $\mathbf{i} \in e_q(L(\tilde{\mathcal{I}}))$, then for $r > qd$, $i_r = i_{\pi(r)}$, which implies $y_{\mathbf{i}_{K_\beta}}^{(\beta)} = x_{\mathbf{i}_{K_\beta\cap[qd]}}^{(\beta)} = x_{\mathbf{i}_{J_\beta}}^{(\beta)}$. As this holds true for any collection $x^{(\beta)}$, we obtain \eqref{norm2}.

Finally, we prove
\begin{equation}\label{finalstep}
\lVert \tilde{S}^{\abs{\mathbf{k}}} \rVert_{\mathcal{K}} \lesssim \lVert A_{\abs{\mathbf{k}}} \rVert_{\mathcal{K}}
\end{equation}
for any partition $\mathcal{K} \in P_{q\abs{\mathbf{k}}}$. To see this, note that if $\mathbf{i} \in L(\tilde{\mathcal{I}})$, we have
$\tilde{s}_\mathbf{i}^{\abs{\mathbf{k}}} = a_\mathbf{i}^{\abs{\mathbf{k}}} \prod_{\alpha=1}^\nu \mathbb{E} X_{i_\alpha}^{k_\alpha - l_\alpha}$. As a consequence,
$$\tilde{S}^{\abs{\mathbf{k}}} = (A_{\abs{\mathbf{k}}} \circ 1_{L(\tilde{\mathcal{I}})}) \circ \otimes_{\alpha = 1}^{\abs{\mathbf{k}}} v_\alpha,$$
where the vectors $v_\alpha$ are defined by $v_\alpha = (\mathbb{E} X_i^{k_\alpha - l_\alpha})_{i \le n}$ if $\alpha \in \{\min I_1, \ldots, \min I_\nu \}$ and $v_\alpha = (1, \ldots, 1)$, otherwise. In particular, we always have $\lVert v_\alpha \rVert_\infty \lesssim 1$, and therefore, by Lemma \ref{Absch},
$$\lVert \tilde{S}^{\abs{\mathbf{k}}} \rVert_{\mathcal{K}} \lesssim \lVert A_{\abs{\mathbf{k}}} \circ 1_{L(\tilde{\mathcal{I}})} \rVert_{\mathcal{K}},$$
from where we easily arrive at \eqref{finalstep} by applying Lemma \ref{lemLK}.

Now, \eqref{remtp} follows by combining \eqref{norm2} and \eqref{finalstep}, which finishes the proof.
\end{proof}

\section{The general sub-exponential case: \texorpdfstring{$\alpha \in (0,1]$}{alpha}}\label{section:alpha}
Using slightly different techniques than in the proofs of Theorem \ref{thmch} and Theorem \ref{thmpol}, we may obtain concentration results for polynomials in independent random variables with bounded $\psi_{\alpha}$-norms for any $\alpha \in (0,1]$. Here, the key difference is that we will not compare their moments to products of Gaussians but to Weibull variables.

To this end, we need some more notation. Let $A = (a_\mathbf{i})_{\mathbf{i} \in [n]^d}$ be a $d$-tensor and $I \subset [d]$ a set of indices. Then, for any $\mathbf{i}_I \coloneqq (i_j)_{j \in I}$, we denote by $A_{\mathbf{i}_{I^c}} = (a_\mathbf{i})_{\mathbf{i}_{I^c}}$ the $(d-\abs{I})$-tensor defined by fixing $i_j$, $j \in I$. For instance, if $d = 4$, $I = \{1,3\}$ and $i_1 = 1, i_3 = 2$, then $A_{\mathbf{i}_{I^c}} = (a_{1j2k})_{jk}$. We will also need the notation $P(I^c)$ for the set of all partitions of $I^c$.

For $I = [d]$, i.\,e. we fix all indices of $\mathbf{i}$, we interpret $A_{\mathbf{i}_{I^c}} = a_\mathbf{i}$ as the $\mathbf{i}$-th entry of $A$. Moreover, in this case, we assume that there is a single element $\mathcal{J} \in P(I^c)$ (which we may call the ``empty'' partition), and $\lVert A_{\mathbf{i}_{I^c}} \rVert_{\mathcal{J}} = \abs{a_\mathbf{i}}$ is just the Euclidean norm of $a_\mathbf{i}$. Finally, note that if $I = \emptyset$, $\mathbf{i}_I$ does not indicate any specification, and $A_{\mathbf{i}_{I^c}} = A$.

Using the characterization of the $\Psi_\alpha$ norms in terms of the growth of $L^p$ norms (see Appendix \ref{section:OrliczNorms} for details), \cite[Corollary 2]{KL15} now yields a result similar to Theorem \ref{thmch} for all $\alpha \in (0,1]$:

\begin{korollar}\label{corollary:PsiAlphaKL15}
Let $X_1, \ldots, X_n$ be a set of independent, centered random variables with $\norm{X}_{\Psi_{\alpha}} \le M$ for some $\alpha \in (0,1]$ , $A$ be a symmetric $d$-tensor with vanishing diagonal and consider $f_{d,A}$ as in \eqref{eqn:chaos}. We have for any $t > 0$
\[
	\mathbb{P}\Big( \abs{f_{d,A}(X)} \ge t \Big) \le 2 \exp \left( - \frac{1}{C_{d,\alpha}} \min_{I \subset [d]} \min_{\mathcal{J} \in P(I^c)} \Big( \frac{t}{M^d \max_{\mathbf{i}_I} \norm{A_{\mathbf{i}_{I^c}}}_{\mathcal{J}}} \Big)^{\frac{2\alpha}{2\abs{I}+\alpha\abs{\mathcal{J}}}} \right).
\]
\end{korollar}

The main goal of this section is to generalize Corollary \ref{corollary:PsiAlphaKL15} to arbitrary polynomials similarly to Theorem \ref{thmpol}. This yields the following result:

\begin{satz}\label{thmpolgen}
Let $X_1, \ldots, X_n$ be a set of independent random variables satisfying $\lVert X_i \rVert_{\psi_{\alpha}}\linebreak[3] \le M$ for some $\alpha \in (0,1]$ and $M > 0$. Let $f \colon \mathbb{R}^n \to \mathbb{R}$ a polynomial of total degree $D \in \mathbb{N}$. Then, for any $t > 0$,
{\scriptsize
\begin{align*}
\mathbb{P} (|f(X) - \mathbb{E}f(X)| \ge t) \le 2 \exp\Big(- \frac{1}{C_{D,\alpha}} \min_{1 \le d \le D}\min_{I\subset[d]}\min_{\mathcal{J} \in P(I^c)} \Big(\frac{t}{M^d \max_{\mathbf{i}_I} \lVert (\mathbb{E} f^{(d)}(X))_{\mathbf{i}_{I^c}} \rVert_\mathcal{J}} \Big)^{\frac{2\alpha}{2\abs{I}+\alpha\abs{\mathcal{J}}}}\Big).
\end{align*}
}
\end{satz}

To prove Theorem \ref{thmpolgen}, note that one particular example of centered random variables with $\norm{X}_{\Psi_{\alpha}} \le M$ is given by symmetric Weibull variables with shape parameter $\alpha$ (and scale parameter $1$), i.\,e. symmetric random variables $w$ with $\mathbb{P}(\abs{w}\ge t) = \exp(-t^\alpha)$. In fact, \cite[Example 3]{KL15} especially implies the following analogue of of Lemma \ref{Lat}:
\begin{lemma}\label{wb}
Let $A = (a_\mathbf{i})_{\mathbf{i} \in [n]^d}$ be a $d$-tensor and $(w_i^j)$, $i \le n$, $j \le d$, an array of i.i.d. Weibull variables with shape parameter $\alpha \in (0,1]$. Then, for every $p \ge 2$,
\begin{align*}
\begin{split}
&C_{\alpha,d}^{-1} \sum_{I \subset [d]} \sum_{\mathcal{J} \in P(I^c)} p^{\abs{I}/\alpha+|\mathcal{J}|/2} \max_{\mathbf{i}_I} \lVert A_{\mathbf{i}_{I^c}} \rVert_{\mathcal{J}}\\ \le \ &\lVert \langle A, w^1 \otimes \ldots \otimes w^d \rangle \rVert_p \le C_{\alpha,d} \sum_{I \subset [d]} \sum_{\mathcal{J} \in P(I^c)} p^{\abs{I}/\alpha+|\mathcal{J}|/2} \max_{\mathbf{i}_I} \lVert A_{\mathbf{i}_{I^c}} \rVert_{\mathcal{J}}.
\end{split}
\end{align*}
\end{lemma}

Moreover, we need a replacement of Lemma \ref{comp}. Here, instead of Gaussian random variables we use Weibull random variables to compare the $p$-th moments:

\begin{lemma}\label{comp2}
For any $k \in \mathbb{N}$, any $\alpha \in (0,1]$ and any $p \ge 2$, if $Y_1, \ldots, Y_n$ are independent symmetric random variables with $\lVert Y_i \rVert_{\psi_{\alpha/k}} \le M$, then
$$\Big\lVert \sum_{i = 1}^n a_iY_i \Big\rVert_p \le C_{\alpha,k} M \Big\lVert \sum_{i = 1}^n a_i w_{i_1} \cdots w_{ik} \Big\rVert_p,$$
where $w_{ij}$ are i.i.d. Weibull variables with shape parameter $\alpha$.
\end{lemma}

\begin{proof}
We extend the arguments given in the proof of \cite[Corollary 2]{KL15}. As always, we assume $M=1$. Moreover, note that it suffices to prove Lemma \ref{comp2} for $p \in 2\mathbb{N}$. It follows from Lemma \ref{wb} that $\norm{w_{ij}}_p \ge C_\alpha p^{1/\alpha}$ for any $i,j$, from where we easily arrive at $\norm{w_{i1} \cdots w_{ik}}_p \ge C_{\alpha,k} p^{k/\alpha}$. Consequently, for a set of independent Rademacher variables $\varepsilon_1, \ldots, \varepsilon_n$ which are independent of the $(Y_i)_i$, $\norm{Y_i}_p = \norm{\varepsilon_i Y_i}_p \le C_\alpha p^{k/\alpha} \le C_{\alpha,k} \norm{w_{i1} \ldots w_{ik}}_p$. Therefore, for any $m \in \mathbb{N}$ and using standard symmetrization inequalities,
$$\Big\lVert \sum_{i = 1}^n a_iY_i \Big\rVert_{2m} \le 2 \Big\lVert \sum_{i = 1}^n a_i\varepsilon_i Y_i \Big\rVert_{2m} \le C_{\alpha,k} \Big\lVert \sum_{i = 1}^n a_i w_{i_1} \cdots w_{ik} \Big\rVert_{2m}.$$
\end{proof}

Our next goal is to adapt Lemmas \ref{uncondlemma}, \ref{lemLK} and \ref{Absch} to the ``restricted'' tensors $A_{\mathbf{i}_{I^c}}$. That is, we examine whether (a modification of) the inequality
\begin{equation}\label{uncondrestr}
\lVert (A \circ 1_C)_{\mathbf{i}_{I^c}} \rVert_{\mathcal{J}} \le \lVert A_{\mathbf{i}_{I_c}} \rVert_{\mathcal{J}}
\end{equation}
still holds in this situation, where $\mathcal{J}$ is a partition of $I^c$.

\begin{lemma}\label{generalizations}
Let $A = (a_{\textbf{i}})_{\mathbf{i} \in [n]^d}$ be a $d$-tensor, $I \subset [d]$ and $\mathbf{i}_I \in [n]^I$ fixed.
	\begin{enumerate}
		\item If $C = \{\mathbf{i} \colon i_{k_1} = j_1, \ldots, i_{k_l} = j_l \}$ for some $1 \le k_1 < \ldots < k_l \le d$ (``generalized row''), then \eqref{uncondrestr} holds.
		\item If $C = \{\mathbf{i} \colon i_k = i_l \ \forall k, l \in K \}$ for some $K \subset [d]$ (``generalized diagonal''), then \eqref{uncondrestr} holds.
		\item If $C_1, C_2 \subset [n]^d$ are such that \eqref{uncondrestr} holds, then so is $C_1 \cap C_2$.
		\item If $\mathcal{K} \in P_d$, then
		$\lVert (A \circ 1_{L(\mathcal{K})})_{\mathbf{i}_{I^c}} \rVert_{\mathcal{J}} \le 2^{|\mathcal{K}|(|\mathcal{K}|-1)/2}\lVert A_{\mathbf{i}_{I^c}} \rVert_{\mathcal{J}}$.
		\item For any vectors $v_1, \ldots, v_d \in \mathbb{R}^n$,
		$\lVert (A \circ \otimes_{i=1}^d v_i)_{\mathbf{i}_{I^c}} \rVert_\mathcal{J} \le \lVert A_{\mathbf{i}_{I^c}} \rVert_\mathcal{J} \prod_{i=1}^d \lVert v_i \rVert_\infty$.
	\end{enumerate}
\end{lemma}

\begin{proof}
To see (1), we may assume that $\{k_1, \ldots, k_l\} \cap I = \emptyset$ (note that if $\{k_1, \ldots, k_l\} \cap I \neq \emptyset$, either the conditions are not compatible, in which case $(A \circ 1_C)_{i_{I^c}} = 0$, or we can remove some of the conditions and obtain a subset with $\{k_1, \ldots, k_{\tilde{l}} \} \cap I = \emptyset$).
In this case, if $C$ is a generalized row, then $(A \circ 1_C)_{\mathbf{i}_{I^c}} = A_{\mathbf{i}_{I^c}} \circ 1_{C'}$ for some generalized row $C'$ in $I^c$. This proves (1).

If $C$ is a generalized diagonal, we have to consider two situations. Assuming $K \cap I = \emptyset$, i.\,e. $K$ is subset of $I^c$, we immediately obtain (2). On the other hand, if $K \cap I \neq \emptyset$, then $(A \circ 1_C)_{\mathbf{i}_{I^c}} = A_{\mathbf{i}_{I^c}} \circ 1_{C'}$ for some generalized \emph{row} $C'$ in $I^c$, readily leading to (2) again.

(3) is clear. To see (4), one may argue as in the proof of Lemma \ref{lemLK} (for $q=1$), replacing Lemma \ref{uncondlemma} (2) and (3) by their analogues we just proved. Finally, an easy modification of the proof of Lemma \ref{Absch} yields (5).
\end{proof}

We are now ready to prove Theorem \ref{thmpolgen}. Here, we recall the notation used in the proof of Theorem \ref{thmpol}, with the only difference that now, by $f \lesssim g$ we mean an inequality of the form $f \le C(D,\alpha) g$, where $C(D,\alpha)$ may depend on $D,\alpha$.

\begin{proof}[Proof of Theorem \ref{thmpolgen}]
We will follow the proof of Theorem \ref{thmpol}. In particular, let us assume $M = 1$.

\textbf{Step 1.} Recall the inequality
$$|f(X) - \mathbb{E}f(X)| \le \sum_{d=1}^D \sum_{\nu=1}^d \sum_{\textbf{k} \in I_{\nu,d}} \Big|\sum_{\textbf{i} \in [n]^{\underline{\nu}}} a_{\textbf{i}}^{\textbf{k}} (X_{i_1}^{k_1}- \mathbb{E}X_{i_1}^{k_1}) \cdots (X_{i_\nu}^{k_\nu}-\mathbb{E}X_{i_\nu}^{k_\nu})\Big|$$
from the proof of Theorem \ref{thmpol}.

\textbf{Step 2.} Applying Lemma \ref{generalizations}, we arrive at
\begin{align*}
&\lVert f(X) - \mathbb{E}f(X) \rVert_p \lesssim \sum_{d=1}^D \sum_{\nu=1}^d \sum_{\textbf{k} \in I_{\nu,d}} \Big\|\sum_{\textbf{i} \in [n]^{\underline{\nu}}} a_{\textbf{i}}^{\textbf{k}} (w_{i_1,1}^{(1)} \cdots w_{i_1,k_1}^{(1)}) \cdots (w_{i_\nu,1}^{(\nu)} \cdots w_{i_\nu,k_\nu}^{(\nu)})\Big\|_p.
\end{align*}
Here, $(w_{i,k}^{(j)})$ is an array of i.i.d. symmetric Weibull variables with shape parameter $\alpha$. Now we may define $d$-tensors $A_d$ as in the proof of Theorem \ref{thmpol}. Similarly as in \eqref{eqn:LpNormEstimate}, rewriting and applying Lemma \ref{wb} together with Lemma \ref{generalizations} (4) then yields
{\scriptsize
\begin{align*}
\lVert f(X) - \mathbb{E}f(X) \rVert_p &\lesssim \sum_{d=1}^D \sum_{\nu=1}^d \sum_{\textbf{k} \in I_{\nu,d}}  \lVert \langle A_d \circ 1_{L(\mathcal{K}(k_1, \ldots, k_\nu))}, \otimes_{j=1}^\nu \otimes_{k=1}^{k_j} (w_{i,k}^{(j)})_{i\le n} \rangle \rVert_p\\
&\lesssim \sum_{d=1}^D \sum_{\nu=1}^d \sum_{\textbf{k} \in I_{\nu,d}} \sum_{I \subset [d]} \sum_{\mathcal{J} \in P(I^c)} p^{\abs{I}/r+|\mathcal{J}|/2} \max_{\mathbf{i}_I} \lVert (A_d \circ 1_{L(\mathcal{K}(k_1, \ldots, k_\nu))})_{\mathbf{i}_{I^c}} \rVert_{\mathcal{J}}\\
&\lesssim \sum_{d=1}^D \sum_{I \subset [d]} \sum_{\mathcal{J} \in P(I^c)} p^{\abs{I}/r+|\mathcal{J}|/2} \max_{\mathbf{i}_I} \lVert (A_d)_{\mathbf{i}_{I^c}} \rVert_{\mathcal{J}}.
\end{align*}
}

\textbf{Step 3.} In the proof of Theorem \ref{thmpol} we have decomposed
\begin{align*}
\mathbb{E} \frac{\partial^d f}{\partial x_{i_1} \ldots \partial x_{i_d}}(X) = \nu!l_1! \cdots l_\nu!a_{i_1, \ldots, i_d}  + R_\mathbf{i}^{(d)}
\end{align*}
with a remainder tensor $R_\mathbf{i}^{(d)}$ corresponding to the set of indices $\mathbf{k}$ with $\mathbf{k} > \mathbf{l}$ and $R_\mathbf{i}^{(d)} = 0$ for $d = D$. Again, for any $I \subset [D]$ and any partition $\mathcal{J} \in P(I^c)$,
\begin{align*}
\lVert (A_D)_{\mathbf{i}_{I^c}} \rVert_\mathcal{J} &\le \sum_{\mathcal{K} \in P_D} \lVert (A_D \circ 1_{L(\mathcal{K})})_{\mathbf{i}_{I^c}} \rVert_{\mathcal{J}} \le \sum_{\mathcal{K} \in P_D} \lVert (\mathbb{E} f^{(D)}(X) \circ 1_{L(\mathcal{K})})_{\mathbf{i}_{I^c}} \rVert_{\mathcal{J}}\\ &\lesssim \lVert (\mathbb{E} f^{(D)}(X))_{\mathbf{i}_{I^c}} \rVert_{\mathcal{J}},
\end{align*}
using Lemma \ref{generalizations} (4) in the last step. To complete the proof, we need to show that for any $d = 1, \ldots, D-1$, any $I \subset [d]$ and any partitions $\mathcal{I} \in P([d])$, $\mathcal{J} \in P([d]\backslash I)$,
\begin{equation}\label{tbpg}
\lVert (R^{(d)} \circ 1_{L(\mathcal{I})})_{\mathbf{i}_{I^c}} \rVert_\mathcal{J} \lesssim \sum_{k=d+1}^D \sum_{\substack{\mathcal{K}\in P([k] \backslash I) \\ |\mathcal{K}|\ge|\mathcal{J}|}} \lVert (A_k)_{\mathbf{i}_{I^c}} \rVert_\mathcal{K}.
\end{equation}
Actually, analyzing the proof one can see it is possible to restrict the second sum on the right-hand side to partitions $\mathcal{K}$ with $|\mathcal{K}| \in \{\abs{\mathcal{J}}, \abs{\mathcal{J}}+1 \}$.
Once having proven \eqref{tbpg}, it follows from reverse induction that
{
\footnotesize
\begin{align*}
\sum_{d=1}^D \sum_{I \subset [d]} \sum_{\mathcal{J} \in P(I^c)} p^{\abs{I}/r+|\mathcal{J}|/2} \max_{\mathbf{i}_I} \lVert (A_d)_{\mathbf{i}_{I^c}} \rVert_{\mathcal{J}} \lesssim \sum_{d=1}^D \sum_{I \subset [d]} \sum_{\mathcal{J} \in P(I^c)} p^{\abs{I}/r+|\mathcal{J}|/2} \max_{\mathbf{i}_I} \lVert (\mathbb{E} f^{(d)}(X))_{\mathbf{i}_{I^c}} \rVert_{\mathcal{J}}.
\end{align*}
}Here, we use that for any $p \ge 2$ and any $\abs{\mathcal{K}} \ge \abs{\mathcal{J}}$ we have $p^{\abs{\mathcal{J}}/2} \le p^{\abs{\mathcal{K}}/2}$. In view of Step 2 and Proposition \ref{normtoconc}, this finishes the proof.

\textbf{Step 4.} One last time we need to recall some definitions from the proof of Theorem \ref{thmpol}. We fix some $I \subset [d]$ and $\mathbf{i}_I$, an admissible partition $\mathcal{I} \in P([d])$ and some associated extension $\tilde{\mathcal{I}} \in P([k])$, the $d$-tensor $S_\mathcal{I}^{(d,\mathbf{k})} = (s_\mathbf{i}^{(d,k_1, \ldots, k_\nu)})_{\mathbf{i} \in [n]^d} = (s_\mathbf{i}^{(d)})_{\mathbf{i} \in [n]^d}$ and for any $\mathbf{k} \in I_{\nu,\le D}$ with $\mathbf{k} > \mathbf{l}$ a $\abs{\mathbf{k}}$-tensor $\tilde{S}^{\abs{\mathbf{k}}}$. The notion of admissibility was not relevant in Theorem \ref{thmpol}, as we have not fixed any indices $I$ and values $\mathbf{i}_{I} \in [n]^I$. Here, it simply means that the level sets have to compatible with the fact that we have fixed some of the partial derivatives by $I$ and $\mathbf{i}_I$. Also, note that $\mathcal{I}$ is a partition of $[d]$ and not of $[d] \backslash I$, since it arises from level sets of partial derivatives and includes the partial derivatives taken in $I$.

Our aim is to prove that there is a partition $\mathcal{K} \in P([k]\backslash I)$ with $|\mathcal{K}| \in \{ |\mathcal{J}|, |\mathcal{J}|+1 \}$ such that
\begin{equation}\label{remtpgen}
\lVert (S_\mathcal{I}^{(d,\mathbf{k})})_{\mathbf{i}_{I^c}} \rVert_\mathcal{J} \lesssim \lVert (A_{\abs{\mathbf{k}}})_{\mathbf{i}_{I^c}} \rVert_{\mathcal{K}}.
\end{equation}
$\mathcal{K} = \{K_1, \ldots, K_{\mu+1} \}$ will be defined as follows: for $j = 1,\ldots,\mu$, we add all elements of $J_j$ to $K_j$, so that it remains to assign the elements $r \in \{d+1, \ldots, \abs{\mathbf{k}}\}$ to the sets $K_j$. Since $\tilde{\mathcal{I}}$ is a partition of $\abs{\mathbf{k}}$, there is a unique $k \in \{1,\ldots, \nu\}$ such that $r \in \tilde{I}_k$. Take the smallest element $t =: \pi(r)$ in $\tilde{I}_k$ (since $I_k \subset \tilde{I}_k$, we have $t \in [d]$). If $t \in I^c$, it follows that $t \in K_j$ for some set $K_j$ and we add $r$ to $K_j$. If $t \in I$, we assign $r$ to an ``extra set'' $K_{\mu+1}$. In particular, it may happen that $K_{\mu+1} = \emptyset$. In this case, we ignore $\beta = \mu +1$ in the rest of the proof.

\begin{center}
\begin{figure}[!ht]
\scalebox{0.75}{\includegraphics{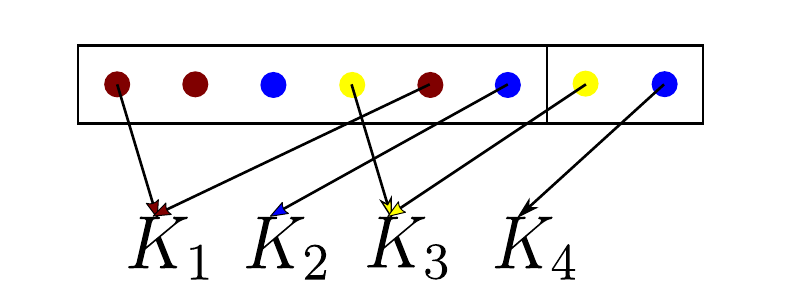}}
\caption{An illustration of the procedure of producing the partition $\mathcal{K}$; here, $I = \{1,2\}$ and we used colors to indicate the partition, i.\,e. $\tilde{\mathcal{I}} = \{ \{1,2,5\}, \{ 3,6,8 \}, \{4,7\} \}$. $\{8\}$ belongs to $K_4 = K_{\mu+1}$ since $\{ 3 \} \in I$. Changing its color to yellow would produce a partition $\mathcal{K}$ with 3 subsets.}
\end{figure}
\end{center}

First off, we claim
\begin{equation}\label{norm2gen}
\lVert (S_\mathcal{I}^{(d,\abs{\mathbf{k}})})_{\mathbf{i}_{I^c}} \rVert_\mathcal{J} \le \lVert (\tilde{S}^{\abs{\mathbf{k}}})_{\mathbf{i}_{I^c}} \rVert_\mathcal{K}.
\end{equation}
To see \eqref{norm2gen}, let $x = (x^{(\beta)})_{\beta=1,\ldots,\mu} = ((x^{(\beta)}_{\mathbf{i}_{J_\beta}}))_{\beta = 1,\ldots,\mu}$, be such that $$\abs{x}_{\mathcal{J}} = \max_{\beta =1,\ldots,\mu} \norm{x^{(\beta)}}_2 \le 1.$$
We embed this in the unit ball with respect to $\abs{x}_{\mathcal{K}} = \max_{i = 1,\ldots,\mu+1} \norm{x^{\beta}}_2$ by defining $y = (y^{(\beta)})_{\beta = 1,\ldots, \mu+1}$ via
\[
y^{(\beta)}_{\mathbf{i}_{K_\beta}} = 
\begin{cases} 
x^{(\beta)}_{\mathbf{i}_{K_\beta \cap [d]}} \prod_{r\in K_\beta\setminus[d]} 1_{i_r = i_{\pi(r)}} & \beta = 1,\ldots, \mu \\
\prod_{r \in K_{\mu+1}} 1_{i_r = i_{\pi(r)}} & \beta = \mu+1.
\end{cases}
\]
Note that $y^{(\mu+1)}$ only has a single non-zero element, and thus it is easy to see that $\abs{y}_{\mathcal{K}} \le 1$. Moreover, by the definition of the matrix $\tilde{S}^{\abs{k}}$ and the fact that if $\mathbf{i} \in L(\tilde{\mathcal{I}})$, then for $r > d$, $i_r = i_{\pi(r)}$, which implies $y_{\mathbf{i}_{K_\beta}}^{(\beta)} = x_{\mathbf{i}_{K_\beta\cap[d]}}^{(\beta)} = x_{\mathbf{i}_{J_\beta}}^{(\beta)}$ as well as $y_{\mathbf{i}_{K_{\mu+1}}}^{(\mu+1)} = 1$ we have
\begin{align}
\skal{(S^{(d,\mathbf{k})})_{\mathbf{i}_{I^c}}, \bigotimes_{\beta = 1}^\mu x^{(\beta)}} = \skal{(\tilde{S}^{(\abs{k})})_{\mathbf{i}_{I^c}}, \bigotimes_{\beta = 1}^{\mu+1} y^{(\beta)}}.
\end{align}
Hence, the supremum on the left hand side of \eqref{norm2gen} is taken over a subset of the unit ball with respect to $\abs{x}_{\mathcal{K}}$.

Finally, it remains to prove
\begin{equation}\label{finalstepgen}
\lVert (\tilde{S}^{\abs{\mathbf{k}}})_{\mathbf{i}_{I^c}} \rVert_{\mathcal{K}} \lesssim \lVert (A_{\abs{\mathbf{k}}})_{\mathbf{i}_{I^c}} \rVert_{\mathcal{K}}
\end{equation}
for any partition $\mathcal{K} \in P(I^c)$. This may be achieved as in the proof of Theorem \ref{thmpol}, replacing Lemma \ref{Absch} by Lemma \ref{generalizations} (5).

Combining \eqref{norm2gen} and \eqref{finalstepgen} yields \eqref{remtpgen}, which finishes the proof.
\end{proof}

It remains to prove Proposition \ref{proposition:weakFormHansonWright} and Theorem \ref{wHSn} (from which Corollary \ref{wHSnc} follows immediately).

\begin{proof}[Proof of Proposition \ref{proposition:weakFormHansonWright}]
The case $\alpha \in (0,1]$ follows immediately from the $d = 2$ case of Corollary \ref{corollary:PsiAlphaKL15}. $\alpha = 2$ corresponds to the well-known Hanson--Wright inequality, see e.\,g. \cite{RV13}.
\end{proof}

\begin{proof}[Proof of Theorem \ref{wHSn}]
Let $\alpha \in (0,1]$ and consider the bound given by Theorem \ref{thmpolgen}. Fix any $d = 1, \ldots, D$. Then, for any $I \subset [d]$, any $\mathbf{i}_I$ and any $\mathcal{J} \in P(I^c)$, we have
$$
\lVert (\mathbb{E} f^{(d)}(X))_{\mathbf{i}_{I^c}} \rVert_{\mathcal{J}} \le \lVert (\mathbb{E} f^{(d)}(X))_{\mathbf{i}_{I^c}} \rVert_{\mathrm{HS}} \le \lVert \mathbb{E} f^{(d)}(X) \rVert_{\mathrm{HS}}
$$
(using \eqref{pHS}) as well as
$$\frac{\alpha}{d} \le \frac{2\alpha}{2\abs{I}+\alpha\abs{\mathcal{J}}} \le 2.$$
If $t/(M^d \lVert \mathbb{E} f^{(d)}(X) \rVert_\mathrm{HS}) \ge 1$, this immediately yields the result. Otherwise, note that the tail bound given in Theorem \ref{wHSn} is trivial. (In fact, here one needs to ensure that $C_{D, \alpha}$ is sufficiently large, e.\,g. $C_{D, \alpha} \ge 1$. It is not hard to see that in general this condition will be satisfied anyway.)\par
In a similar way, it is possible to derive the same results for $\alpha = 2/q$ and any $q \in \mathbb{N}$ from Theorem \ref{thmpol}.\par
From these results, the exponential moment bound follows by standard arguments, see for example \cite[Proof of Theorem 1.1]{BGS18}.
\end{proof}

\appendix
\section{Properties of Orlicz quasinorms}\label{section:OrliczNorms}
As mentioned in the introduction, Orlicz norms \eqref{eqn:definitionOrliczNorm} satisfy the triangle inequality only for $\alpha \ge 1$. However, for any $\alpha \in (0,1)$ \eqref{eqn:definitionOrliczNorm} still is a quasinorm, which for many purposes is sufficient. We shall collect some elementary results on Orlicz quasinorms in this appendix. The first result is a H{\"o}lder-type inequality for the $\Psi_{\alpha}$ norms.

\begin{lemma}\label{lemma:HoelderTypeInequality}
Let $X_1, \ldots, X_k$ be random variables such that $\norm{X_i}_{\Psi_{\alpha_i}} < \infty$ for some $\alpha_i \in (0,1]$ and let $t \coloneqq (\sum_{i = 1}^k \alpha_i^{-1})^{-1}$. Then $\norm{\prod_{i = 1}^k X_i}_{\Psi_t} < \infty$ and
\[
\Big\lVert \prod_{i = 1}^k X_i \Big\rVert_{\Psi_{t}} \le \prod_{j = 1}^k \norm{X_i}_{\Psi_{\alpha_i}}.
\]
\end{lemma}

\begin{proof}
By homogeneity we can assume $\norm{X}_{\Psi_{\alpha_i}} = 1$ for all $i = 1,\ldots,k$. We will need the general form of Young's inequality, i.\,e. for all $p_1, \ldots, p_k > 1$ satisfiyng $\sum_{i = 1}^k p_i^{-1} = 1$ and any $x_1, \ldots, x_k \ge 0$ we have
\[
\prod_{i = 1}^k x_i \le \sum_{i = 1}^k p_i^{-1} x_i^{p_i},
\]
which follows easily from the concavity of the logarithm. If we apply this to $p_i \coloneqq \alpha_i t^{-1}$ and use the convexity of the exponential function, we obtain
\begin{align*}
\IE \exp\Big( \prod_{i = 1}^k \abs{X_i}^{t} \Big) &\le \IE \exp \Big( \sum_{j = 1}^k p_i^{-1} \abs{X_i}^{\alpha_i} \Big) \le \sum_{j = 1}^k p_i^{-1} \IE \exp \Big( \abs{X_i}^{\alpha_i} \Big) \le 2.
\end{align*}
Consequently, we have $\norm{\prod_{i = 1}^k X_i}_{\Psi_t} \le 1$.
\end{proof}

The random variables $X_1, \ldots, X_k$ need not be independent, i.\,e. we can consider a random vector $X = (X_1,\ldots, X_k)$ with marginals having $\alpha$-sub-exponential tails. The special case $\alpha_i = \alpha$ for all $i = 1,\ldots, k$ gives
\[
\Big \lVert \prod_{i = 1}^k X_i \Big \rVert_{\Psi_{\alpha/k}} \le \prod_{j = 1}^k \norm{X_i}_{\Psi_\alpha}.
\]

To state the other lemmas, for any $0 < \alpha < 1$ define
\begin{equation}
d_\alpha \coloneqq (\alpha e)^{1/\alpha}/2 \qquad \text{and} \qquad D_\alpha \coloneqq (2e)^{1/\alpha}.
\end{equation}

\begin{lemma}\label{lemma:PsiAlphaAndGrowth}
For any $0 < \alpha < 1$ we have
\begin{equation}	\label{eqn:EquivalenceOfNorms}
d_\alpha \sup_{p \ge 1} \frac{\norm{X}_p}{p^{1/\alpha}} \le \norm{X}_{\Psi_\alpha} \le D_\alpha \sup_{p \ge 1} \frac{\norm{X}_p}{p^{1/\alpha}}.
\end{equation}
\end{lemma}

The statement of the lemma remains true for $\alpha \ge 1$, with ($\alpha$-independent constants) $d_\alpha = 1/2$ and $D_\alpha = 2e$, see \cite[Section 8]{Bob10}. In the proof, we will closely follow the proof therein, but keep track of the $\alpha$-dependent constants.

\begin{proof}
We begin with the left inequality. By homogeneity, we assume $\norm{X}_{\Psi_\alpha} = 1$. First let us show that we have 
\begin{equation}		\label{eqn:gGe0}
g(x) \coloneqq \left( \alpha e  \right)^{-1/\alpha} e^{x^\alpha} - x \ge 0 \quad \quad \text{for}\quad  x \ge 0.
\end{equation}
Note that $g$ is continuous on $[0,\infty)$ and differentiable on $(0,\infty)$ with $g(0) > 0$ and $g(x) \to \infty$ as $x \to \infty$. Therefore, it suffices to find the critical points. We can rewrite the condition $g'(x) = 0$ as $e^y y = y^{1/\alpha} (\alpha e)^{1/\alpha}$, setting $y \coloneqq x^\alpha$. From this representation it can be seen that there can be at most two points $x_0$ and $x_1$ satisfying this condition. One of these points is $x_\alpha \coloneqq \alpha^{-1/\alpha}$, and we have $g(x_\alpha) = 0$. A short calculation shows that $g''(x_\alpha) = \alpha^{1/\alpha +1} > 0$, so that $x_\alpha$ is a global minimum, from which $g \ge 0$ follows. \par
Next, from this we can infer for all $p \ge 1$ and $\alpha > 0$
\begin{equation}
x^p \le \left( \frac{p}{\alpha e} \right)^{p/\alpha} e^{x^\alpha}.
\end{equation}
Indeed, by a transformation $y = x^p$ and the change $\tilde{\alpha} = \frac{\alpha}{p}$ this is just an application of \eqref{eqn:gGe0}. Consequently, for any $p \ge 1$ we have
\[
\mathbb{E} \abs{X}^p \le \left( \frac{p}{\alpha e} \right)^{p/\alpha} \IE \exp\left( \abs{X}^\alpha \right) \le 2 \left( \frac{p}{\alpha e} \right)^{p/\alpha} \le 2^p \left( \frac{p}{\alpha e} \right)^{p/\alpha},
\]
i.\,e.
\[
\norm{X}_p \le 2 (\alpha e)^{-1/\alpha} p^{1/\alpha}.
\]
For the second inequality, again assume that $\sup_{p \ge 1} \frac{\norm{X}_p}{p^{1/\alpha}} = 1.$
First, we need to extend the the supremum to $p \in [\alpha, \infty)$, which can be done as follows. For any $p \in [\alpha, 1)$ we have
\[
\frac{\norm{X}_p}{p^{1/\alpha}} \le \frac{\norm{X}_1}{p^{1/\alpha}} \le \frac{1}{p^{1/\alpha}} \le \frac{1}{\alpha^{1/\alpha}}
\]
and therefore
\[
\sup_{p \ge \alpha} \frac{\norm{X}_p}{p^{1/\alpha}} \le \frac{1}{\alpha^{1/\alpha}}.
\]
Now, by Taylor's expansion and using the inequality $n^n \le e^n n!$ this gives
\[
\IE \exp\left( \frac{\abs{X}^\alpha}{t^\alpha} \right) = 1 + \sum_{n = 1}^\infty \frac{\IE \abs{X}^{\alpha n}}{t^{\alpha n} n!} \le 1 + \sum_{n = 1}^\infty \frac{n^n}{n! t^{\alpha n}} \le 1 + \sum_{n = 1}^\infty \left(\frac{e}{t^\alpha}\right)^{n} = \frac{1}{1 - e/t^\alpha}.
\]
For $t = (2e)^{1/\alpha}$ this is less or equal to $2$, so that
\[
\norm{X}_{\Psi_\alpha} \le (2e)^{1/\alpha} \sup_{p \ge 1} \frac{\norm{X}^p}{p^{1/\alpha}}.
\]
\end{proof}

\begin{lemma}
For any $0 < \alpha < 1$ and any random variables $X,Y$ we have
\begin{equation}
  \norm{X + Y}_{\Psi_\alpha} \le 2^{1/\alpha}\left( \norm{X}_{\Psi_\alpha} + \norm{Y}_{\Psi_\alpha} \right).
\end{equation}
\end{lemma}

\begin{proof}
Let $K \coloneqq \norm{X}_{\Psi_\alpha}$ and $L \coloneqq \norm{Y}_{\Psi_\alpha}$ and define $t \coloneqq 2^{1/\alpha}(K+L)$. We have
\begin{align*}
\IE \exp\left( \frac{\abs{X+Y}^\alpha}{t^\alpha} \right) &\le \IE \exp \left( \frac{(\abs{X} + \abs{Y})^\alpha}{t^\alpha} \right) \le \IE \exp \left( \frac{\abs{X}^\alpha + \abs{Y}^\alpha}{2(K+L)^\alpha} \right) \\
&\le \IE \exp \left( \frac{\abs{X}^\alpha}{2K^\alpha} \right) \exp \left( \frac{\abs{Y}^\alpha}{2L^\alpha} \right) \\
&\le \frac{1}{2} \IE \exp\left( \frac{\abs{X}^\alpha}{K^\alpha} \right) + \frac{1}{2} \IE \exp\left( \frac{\abs{Y}^\alpha}{L^\alpha} \right) \le 2.
\end{align*}
Here, the second step follows from the inequality $(x+y)^\alpha \le x^\alpha + y^\alpha$ valid for all $x,y \ge 0$ and $\alpha \in [0,1]$, and the fourth one is an application of Young's inequality $ab \le a^2/2 + b^2/2$ for all positive $a,b$.
\end{proof}

\begin{lemma}\label{lemma:ExpectationVsRV}
Let $0 < \alpha < 1$. For all random variables $X$ we have
\[
\norm{\IE X}_{\Psi_\alpha} \le \frac{1}{d_\alpha (\log 2)^{1/\alpha}} \norm{X}_{\Psi_\alpha}.
\]
\end{lemma}

\begin{proof}
Assuming $\norm{X}_{\Psi_\alpha} < \infty$, an application of Lemma \ref{lemma:PsiAlphaAndGrowth} gives
\[
  \norm{\IE X}_{\Psi_\alpha} = \frac{\abs{\IE X}}{(\log2)^{1/\alpha}} \le \frac{\norm{X}_1}{(\log2)^{1/\alpha}} \le \frac{1}{d_\alpha (\log2)^{1/\alpha}}\norm{X}_{\Psi_\alpha}.
\]
\end{proof}

As a consequence of the last two results, we can readily infer the following corollary.

\begin{korollar}\label{corollary:CenteringAndPsiAlpha}
For any $\alpha > 0$ and any random variable $X$ we have
\[
\norm{X - \IE X}_{\Psi_\alpha} \le 2^{1/\alpha}\left(1+(d_\alpha \log2)^{-1/\alpha} \right) \norm{X}_{\Psi_\alpha}.
\]
\end{korollar}

\printbibliography
\end{document}